\UseRawInputEncoding
\documentclass[11pt,twoside,a4paper]{amsart}
\usepackage{amssymb}

\usepackage{amsmath}
\usepackage{amsthm,amssymb,latexsym}
\usepackage[all,cmtip]{xy}
\usepackage{url,ae}
\usepackage{t1enc}
\usepackage{fancyheadings}
\usepackage{mathrsfs}
\usepackage{graphicx}
\usepackage{eufrak} 
\usepackage{geometry}

\date{\today}

\def\1{{\bf 1}}
\def\div{{\rm div}}

\def\la{\langle}
\def\ra{\rangle}

\def\w{\wedge}

\def\dbar{\bar\partial}

\def\C{{\mathbb C}}
\def\w{{\wedge}}

\def\Cu{{\mathcal C}}

\def\codim{{\rm codim\,}}

\def\rank{{\rm rank\, }}

\def\E{{\mathcal E}}

\def\Q{{\mathbb Q}}

\def\U{{\mathcal U}}

\def\J{{\mathcal J}}
\def\nbh{neighborhood }

\def\be{\begin{equation}}
\def\ee{\end{equation}}

\def\Ok{\mathcal O}

\def\Pr{{\mathbb P}}

\def\PM{{\mathcal {PM}}}

\def\pmm{{pseudomeromorphic }}

\def\eo{{\epsilon\to 0}}
\def\e{{\epsilon}}

\newtheorem{thm}{Theorem}[section]
\newtheorem{lma}[thm]{Lemma}
\newtheorem{cor}[thm]{Corollary}
\newtheorem{prop}[thm]{Proposition}

\theoremstyle{definition}

\newtheorem{df}[thm]{Definition}

\theoremstyle{remark}

\newtheorem{preremark}[thm]{Remark}
\newtheorem{preex}[thm]{Example}

\newenvironment{remark}{\begin{preremark}}{\qed\end{preremark}}
\newenvironment{ex}{\begin{preex}}{\qed\end{preex}}

\numberwithin{equation}{section}

\usepackage{xcolor}

\title[On proper intersections on a singular analytic space]{On proper intersections on a
singular analytic space}

\begin{document}

\date{\today}

\author[Mats Andersson \& H\aa kan Samuelsson Kalm]{Mats Andersson \& H\aa kan Samuelsson Kalm} 

\address{Department of Mathematical Sciences\\Chalmers University of Technology and University of
Gothenburg\\SE-412 96 G\"OTEBORG\\SWEDEN}

\email{matsa@chalmers.se, hasam@chalmers.se}


\thanks{The first author  was
partially supported by the Swedish
Research Council}

\begin{abstract}
Given a reduced analytic space $Y$ we introduce a class of {\it nice} cycles, including
all effective $\Q$-Cartier divisors. Equidimensional nice cycles that intersect properly allow for a natural intersection product.
Using $\dbar$-potentials and residue calculus we provide an intrinsic way of defining this product. 
The intrinsic definition makes it possible to prove global formulas. In case $Y$ is smooth all
cycles are differences of nice cycles, and so we get a new way to define classical proper intersections.
\end{abstract}


\maketitle

\section{Introduction}

Let $Y$ be a complex manifold of dimension $n$.
In this paper a cycle $\mu$ in $Y$ is a locally finite sum $\sum_j a_jZ_j$, where $a_j\in\mathbb{Q}\setminus\{0\}$ and $Z_j$ are 
distinct irreducible
analytic subsets of $Y$. The cycle is effective if $a_j>0$.
The support, $|\mu|$, of $\mu$ is the union of the $Z_j$ and the (co)dimension of $\mu$ is the 
(co)dimension of $\cup_jZ_j$.
Assume that $\mu_1,\ldots,\mu_r$ are cycles in $Y$ of
pure codimensions. 
It is well-known that then
\begin{equation}\label{codimolikhet}
\text{codim}\big(|\mu_1|\cap\cdots\cap|\mu_r|\big) \leq \text{codim}\, |\mu_1|+\cdots +\text{codim}\, |\mu_r|.
\end{equation}
If equality holds in \eqref{codimolikhet}, then $\mu_1,\ldots,\mu_r$ are said to intersect properly. In that case there
is a well-defined cycle 
\begin{equation*}
\mu_r\cdot  \cdots  \cdot\mu_1=\sum_j m_j V_j,
\end{equation*}
the proper intersection product,
where $V_j$ are the irreducible components of the set-theoretical intersection
$|\mu_1|\cap\cdots\cap|\mu_r|$  and $m_j\in\mathbb{Q}$. In general some $m_j$ may be $0$. However,
if $\mu_j$ are effective, then $\mu_r\cdots\mu_1$ is effective and all $m_j>0$.

 
Classically, this intersection product was  defined algebraically, see, e.g., \cite{Fult}.  
On the analytic side, given a cycle $\mu=\sum_ja_jZ_j$, recall that there is an associated closed current $[\mu]=\sum_ja_j[Z_j]$, 
the Lelong current of $\mu$, where $[Z_j]$ is integration over the regular part of $Z_j$.  
It is quite remarkable that one can define the Lelong current of the
proper intersection as
\begin{equation}\label{basal}
[\mu_r\cdot  \cdots \, \cdot\mu_1]=[\mu_r]\w \cdots\w [\mu_1],
\end{equation}
where 
the current product on the right-hand side is given a meaning via suitable regularizations of the various factors, 
see, e.g., \cite{Ch}.

\smallskip
Let us now assume that $Y$ is a reduced analytic space of pure dimension $n$. Then
there is no known general analogous intersection theory; not even
\eqref{codimolikhet} holds in general.  
We say that an effective cycle $\mu$ in $Y$ is {\it nice} if locally there is an embedding $i\colon Y\to Y'$,
where $Y'$ is smooth, and an effective cycle $\mu'$ in $Y'$ such that $i_*\mu$, i.e., $\mu$ considered as a cycle in $Y'$, is the 
proper intersection $\mu'\cdot i_*Y$ in $Y'$.  
Since such embeddings are essentially unique, being nice is an intrinsic property
in $Y$.  In general, the irreducible components of a nice cycle are not nice. 
If $\mu_1,\ldots,\mu_r$ are nice, then \eqref{codimolikhet} holds, and as in the smooth case we say that 
they intersect properly if equality holds. In that case we have 
an intrinsic nice cycle $\mu=\mu_r\cdot  \cdots \, \cdot\mu_1$. As in the smooth case, $\mu_r\cdot  \cdots \, \cdot\mu_1$ is commutative.
If $i\colon Y\to Y'$ is a local embedding and $\mu_j'$ 
are effective cycles in $Y'$  such that $i_*\mu_j=\mu_j'\cdot i_*Y$, then $\mu_1',\ldots,\mu'_r,i_*Y$ intersect properly and
$i_*\mu=\mu_r'\cdot  \cdots \, \cdot\mu_1'\cdot i_*Y$. 
 See Section~\ref{nicecycles} for details.

It is well-known that the proper intersection in a complex manifold of cycles with integer coefficients has integer coefficients.
However, in a reduced complex space we need to 
consider cycles with rational coefficients. In fact, even if $\mu$ has integer coefficients,
$\mu'$ may need to have rational coefficients. 
Moreover, the intersection product of nice cycles with integer coefficients in general has rational coefficients.

The intersection theory for nice cycles in $Y$ is in a way a quite simple consequence of the
proper intersection theory in ambient space.  Our first main result is an intrinsic way to define
proper intersections in $Y$, i.e., with no explicit reference to any ambient space.  
The approach in \cite{Ch} to use regularizations of $[\mu_j]$
seems to be difficult to extend when $Y$ is singular. 
Instead we introduce \emph{good} $\dbar$-potentials. If $\mu$ is a cycle in $Y$ we say that 
a current $u$ in $Y$ is a good $\dbar$-potential of $\mu$ if $\dbar u=[\mu]$, $u$ is smooth outside $|\mu|$, and $u$
is pseudomeromorphic in $Y$. This last requirement is a regularity property that will be explained below.

\begin{thm}\label{nymain1}
Let $Y$ be a reduced analytic space of pure dimension.
\begin{itemize}
\item[(i)] Each nice cycle in $Y$ locally has a good $\dbar$-potential.
\item[(ii)] If $\mu_1$ is a nice cycle and $u_2$ is a good 
$\dbar$-potential of a nice cycle $\mu_2$, then $u_2\wedge [\mu_1]$, a priori defined outside $|\mu_2|$, has 
a unique pseudomeromorphic extension to $Y$.
\item[(iii)] If $\mu_1$ and $\mu_2$ intersect properly, then the product
$ 
[\mu_2]\wedge[\mu_1]:=\dbar(u_2\wedge [\mu_1])
$
 is the Lelong current of the nice cycle $\mu_2\cdot\mu_1$.
\end{itemize}
\end{thm}

In particular, the product in (iii) is commutative and independent of the choice of $u_2$. 
This follows also quite easily from the intrinsic definition and residue calculus, see Proposition~\ref{snoddas1} below.

If $\mu_1,\ldots,\mu_r$ are nice and intersect properly, then in a neighborhood of any $x\in|\mu_1|\cap\cdots\cap |\mu_r|$
this result can be iterated to give the Lelong current of $\mu_r\cdot \cdots \cdot \mu_1$ there.
Since all effective cycles in a manifold are nice this gives in particular a new definition of proper intersection when $Y$ is smooth.

\begin{remark}\label{tejprulle}
It might look more natural to use $dd^c$-potentials rather than $\dbar$-potentials. However,
the latter choice gives access to residue theory, without which we cannot show existence of local potentials, let alone 
define the product $u_2\wedge[\mu_1]$.
\end{remark}
 
Our intrinsic definition of the proper intersection of nice cycles makes it possible to prove global results for a compact 
singular $Y$. 
 
\begin{thm}\label{main2}
Let $Y$ be compact, $\omega$ a K\"ahler form, and let
$\mu_1,\mu_2,\ldots,\mu_r$ be nice cycles in $Y$ of codimensions $\kappa_1,\ldots,\kappa_r$, respectively.
Assume that $\mu_1,\ldots,\mu_k$ intersect properly for $k=2,\ldots,r$.  
Assume also that for each $j=1,\ldots, r$ there is a smooth (closed) form $\alpha_j$, and a \pmm  current $a_j$ 
of bidegree $(\kappa_j,\kappa_j-1)$,
smooth in $Y\setminus |\mu_j|$, 
such that
\begin{equation}\label{rutger}
\dbar a_j= [\mu_j]-\alpha_j, \  j=1,\ldots,r.
\end{equation}
In addition, suppose that all of the $\alpha_j$, except possibly $\alpha_1$, locally have smooth $\dbar$-potentials.
If $\kappa=\kappa_1+\cdots +\kappa_r$, then  
\begin{equation}\label{pucko}
\int_Y [\mu_r]\wedge\cdots \wedge[\mu_1]\w \omega^{n-\kappa}=
 \int_Y \alpha_r\w\cdots\wedge\alpha_1\w \omega^{n-\kappa}.
 \end{equation}
 \end{thm}
 
Formula \eqref{pucko} suggests that the intersection product is ``cohomologically sound''. 
 
\smallskip

We say that a cycle $\mu$ in $Y$ is an RE-cycle if it is a
locally finite linear combination with positive rational coefficients
of fundamental cycles of so-called
regular embeddings, that is, locally complete intersection ideals in $\Ok_Y$.  
For instance, any effective $\Q$-Cartier divisor is
an RE-cycle.  We will see that all RE-cycles are nice.  
It turns out that  the proper intersection product of RE-cycles is again an
RE-cycle, see Proposition~\ref{kung}.

\begin{ex}\label{bongo} 
Assume that each $\mu_j$ in Theorem~\ref{main2} 
is the fundamental cycle of an ideal defined by a global section of
a Hermitian vector bundle $E_j$ of rank $\kappa_j=\codim \mu_j$.
We will prove in Section~\ref{global} that then there is
a global \pmm current $a_j$, smooth in $Y\setminus |\mu_j|$ such that
$\dbar a_j=da_j=[\mu_j]-c_{\kappa_j}(E_j)$. The  Chern form $c_{\kappa_j}(E_j)$ locally has
a smooth potential; given such a $\mu_j$ we can thus choose $\alpha_j=c_{\kappa_j}(E_j)$ in \eqref{rutger}. 
\end{ex}

There is a proper intersection theory on normal surfaces due to Mumford, \cite{Mumford}.
Recently Barlet and Magn\'usson, \cite{BMmathZ}, defined proper intersections on
a so-called nearly smooth $Y$ by analytic methods. 
In Section~\ref{nearly} we show that for RE-cycles our intersection product coincides with the intersection in \cite{BMmathZ}.

\smallskip
Our approach relies on residue theory, and in 
Section~\ref{prel} we have collected some 
material that we need.
In Section~\ref{frame} we present our $\dbar$-potential approach to proper intersection.
In Section~\ref{glatt} the we use this approach in case $Y$ is smooth and show that it gives the 
usual intersection product.
Proper intersection of nice cycles is discussed in Section~\ref{nicecycles} and Theorem~\ref{nymain1} is proved. 
The special case of RE-cycles is considered in
Section~\ref{REsektion}.
We prove the global Theorem~\ref{main2} in Section~\ref{global}
and provide various examples in Section~\ref{exsection}.
In the last section we show that our product, at least for RE-cycles, coincides with the  
product in \cite{BMmathZ} when $Y$ is nearly  smooth.

\section{Some notions and results in residue theory}\label{prel}
Throughout this section $Y$ is a (reduced)  analytic space of pure dimension $n$.
A smooth form $\alpha$ on $Y_{reg}$ is smooth on $Y$, $\alpha\in\E(Y)$, if 
for a local embedding $i\colon Y\to Y'$ into a manifold $Y'$ there is a smooth form $\tilde\alpha$
in $Y'$ such that $\alpha=i^*\tilde\alpha$ on $Y_{reg}$. It follows that $\dbar$, $d$, and $\partial$
are well-defined on $\E(Y)$. If $X$ is a reduced analytic space and $g\colon X\to Y$ is a holomorphic mapping, then there is 
a functorial pullback mapping $g^*\colon \E(Y)\to \E(X)$, see \cite[Corollary~3.2.21]{BMbook}.

Currents on $Y$ are dual to the compactly supported smooth forms on $Y$. More concretely,
if $i\colon Y\to Y'$ is an embedding, then the currents on $Y$ can be identified with the currents on $Y'$ that vanish on
test forms $\xi$ such that $i^*\xi=0$. Equivalently, the currents on $Y$ can be identified with the currents $\tau$ on $Y'$
such that $\xi\w\tau=0$ for all test forms $\xi$ with $i^*\xi=0$.

If $g\colon X\to Y$ is a proper holomorphic mapping, then there is a pushforward mapping $g_*\colon\Cu(X)\to \Cu(Y)$ from
currents on $X$ to currents on $Y$ defined as $\langle g_*\tau, \xi\rangle=\langle\tau,g^*\xi\rangle$. If $\tau$ is a current on $X$
and $\alpha$ is smooth on $Y$,  then
\begin{equation}\label{bara0}
\alpha \w g_*\tau=g_*(g^*\alpha\w \tau).
\end{equation}
If $Z\subset Y$ is an analytic subset of pure codimension $\kappa$
and $j\colon Z\to Y$ is the inclusion,
then the Lelong current $[Z]$ has bidegree $(\kappa,\kappa)$ and $j_*1=[Z]$.

\smallskip

If $g\colon X\to Y$ is proper,  then there is a mapping 
$g_*\colon \mathcal{Z}(X)\to \mathcal{Z}(Y)$, where $\mathcal{Z}(Y)$ are the cycles in $Y$, 
defined as follows.
Let $\mu\in\mathcal{Z}(X)$ be an irreducible analytic subset of $X$ of dimension $k$. If $\text{dim}\, g(\mu)< k$, then $g_*\mu=0$, and if 
$\text{dim}\, g(\mu)= k$, then $g_*\mu=\text{deg}(g|_\mu) g(\mu)$, where $\text{deg}(g|_\mu)$ is the number of points
in $g^{-1}(x)$ for generic $x\in g(\mu)$. By linearity, $g_*$ extends to $\mathcal{Z}(X)$. We have, cf.\ \cite[Section~2]{aeswy1},
\begin{equation}\label{brittsommar}
g_*[\mu]=[g_*\mu].
\end{equation}

In what follows we will usually identify a cycle with its Lelong current. 
In view of \eqref{brittsommar}, this is consistent with pushforward.
With this convention, $Y$ (considered as a cycle in $Y$) is identified with the constant function $1$,
and if $j\colon Z\to Y$ is an embedding of a reduced analytic space, then $j_*1=[j(Z)]=j(Z)=j_*Z$. 
If $Z$ is an analytic subset of $Y$ and $j$ is the inclusion, then we often identify $Z$ and $j(Z)$.

\subsection{Pseudomeromorphic currents}\label{PMsektion}

The function $1/z^\ell$ in $\C\setminus\{0\}$ extends to $\C$ as a principal value current. The current
$\dbar(1/z^\ell)$ is the associated residue current. If $\U\subset\C^r$ is open, $(z_1,\ldots,z_r)$ are coordinates in $\U$, and 
$\alpha$ is a smooth compactly supported form in $\U$, thus 
\begin{equation*}
\alpha\wedge \frac{1}{z_1^{\ell_1}} \cdots \frac{1}{z_s^{\ell_s}}  \dbar \frac{1}{z_{s+1}^{\ell_{s+1}}}\w\cdots\w\dbar\frac{1}{z_r^{\ell_r}}
\end{equation*}
exists as a tensor product in $\U$. Such a current is an \emph{elementary pseudomeromorphic current}.

A germ of a current $\tau$ at $x\in Y$ is \emph{pseudomeromorphic} 
if it is a finite sum of currents
\begin{equation*}
\pi^1_*\cdots\pi^m_*\nu,
\end{equation*}
where each $\pi^j\colon \U_{j}\to \U_{j-1}$ is either a modification, a simple projection $\U_{j-1}\times Z\to \U_{j-1}$,
or the inclusion of an open subset, and $\nu$ is elementary on $\U_m\subset \C^r$. The set of such germs is an open subset of 
the sheaf of currents on $Y$ and thus is a sheaf, the sheaf $\PM_Y$ of pseudomeromorphic current on $Y$.
This sheaf is closed under $\dbar$ and multiplication by smooth forms.
We refer to, e.g., \cite{AW3} for proofs of the statements below about pseudomeromorphic currents.

If $\tau$ is \pmm and the holomorphic function $h$ vanishes on the support of $\tau$, then $\bar h \tau =0= d\bar h\wedge\tau$.
Furthermore, we have the

\smallskip

\noindent \emph{Dimension principle:} If $\tau$ is pseudomeromorphic, has bidegree $(*,q)$ and support contained in a subvariety of codimension $>q$,
then $\tau=0$. 

\smallskip

\begin{ex}\label{ettgenomf}
If $f$ is a holomorphic function on $Y$, not identically $0$ on any irreducible component, 
then $1/f$, a priori defined outside $Z=f^{-1}(0)$, has a pseudomeromorphic extension to $Y$; cf.\ Example~\ref{ettgenomf2}.
By the dimension principle such an extension must be unique.
The residue current, $\dbar(1/f)$, clearly has support in $Z$.
\end{ex}

If $\tau$ is \pmm in $\U$ and $Z$ is a subvariety, then the natural restriction of $\tau$ to the open
subset $\U\setminus Z$ has a \pmm extension $\mathbf{1}_{\U\setminus Z}\tau$ to $\U$ such that  
\begin{equation}\label{rest0}
\1_{\U\setminus Z}\tau=\lim_{\eo} \chi(|f|^2v/ \e)\tau,
\end{equation}
if $f$ is any tuple of holomorphic functions with $\{f=0\}=Z$, $\chi$ is a smooth function
on $[0,\infty)$ that is $0$ in a neighborhood of $0$ and  $1$ in a neighborhood of $\infty$, and $v$ is a smooth strictly positive function. 
The right-hand side of \eqref{rest0} is indeed independent of the choice of $f$, $\chi$, and $v$.
It follows that
\begin{equation}\label{reimond}
\1_Z\tau:=\tau-\1_{\U\setminus Z}\tau
\end{equation}
is pseudomeromorphic and has support on $Z$. 
If $Z'$ is another subvariety, then
\begin{equation}\label{rest1}
\1_{Z'}\1_Z\tau= \1_{Z\cap Z'}\tau.
\end{equation}
If $\tau$ is \pmm and $\alpha$ is a smooth form, then\footnote{If noting else is suggested by brackets, $\1_Z$ is always assumed to
act on the whole expression on its right.}
 \begin{equation}\label{rest2}
\1_Z \alpha\w \tau= \alpha\w \1_Z\tau.
\end{equation}

If $g\colon X\to Y$ is proper, 
 $\nu$ and $g_*\nu$ are pseudomeromorphic, and  $Z$ is a subvariety of $Y$, then 
\begin{equation}\label{bara1}
\1_Z g_*\nu=g_*(\1_{g^{-1}(Z)} \nu).
\end{equation}

\begin{ex}\label{ettgenomf2}
With the setting in Example~\ref{ettgenomf} we have $\1_{Y\setminus Z}(1/f)=1/f$ by the dimension principle. 
In particular, cf.\ \eqref{rest0}, $\lim_{\epsilon\to 0}\chi(|f|^2v/\epsilon)/f=1/f$.
Thus, $1/f$ can be defined as a principal value current, which is the original definition of Herrera--Lieberman. 
\end{ex}

A current $a$ on $Y$ is \emph{almost semi-meromorphic}, $a\in ASM(Y)$, if $a=\pi_* (\alpha/\sigma)$, where
$\pi\colon X\to Y$ is a modification, $\sigma$ is a section of a line bundle $L\to X$, and $\alpha$ is a smooth form with values in $L$.
Notice that if $L$ is Hermitian and $\beta$ is a smooth form, then $\pi_*(\partial\log |\sigma|^2\w\beta)\in ASM(Y)$. 
The smallest Zariski closed set outside which $a\in ASM(Y)$ is smooth is called the Zariski singular support of $a$.
If $V$ is the Zariski singular support of $a$, then in view of \eqref{bara1} and Example~\ref{ettgenomf2} we have $\1_Va=0$.

\begin{lma}\label{ASMPM}
If $a\in ASM(Y)$ has Zariski singular support $V$ and $\tau\in\PM(Y)$, then $a\wedge \tau$, 
a~priori defined on $Y\setminus V$, has a unique \pmm extension $T$  to 
$Y$ such that  $\1_V T=0$. 
\end{lma}

We write $a\wedge\tau$ for the extension as well. By \eqref{rest0} it follows that if $f$ is a tuple of holomorphic functions
with $\{f=0\}=V$, then $\chi(|f|^2v/\epsilon)a\wedge\tau\to a\wedge\tau$.

 \subsection{Two operations on $\PM$}\label{oponPM}
Let $\sigma$ be a section of a Hermitian holomorphic vector bundle $E\to Y$ with zero set\footnote{$Z$ may contain irreducible components of $Y$.}
$Z$. 

\begin{lma}\label{benin}
For each $k=1,2,\ldots$, there is a (necessarily unique) almost semi-meromorphic current $m_k^\sigma$ in $Y$ that coincides with 
$(2\pi i)^{-1}\partial\log|\sigma|^2\w (dd^c\log|\sigma|^2)^{k-1}$
outside $Z$ and such that $\1_Z m_k^\sigma=0$.
\end{lma}

\begin{proof}
Let $\pi\colon X\to Y$ be a normal modification such that for each connected component of $X$, either $\pi^*\sigma=0$ or
$\pi^*\sigma=\sigma^0 \sigma'$, where 
$\sigma^0$ is a generically non-vanishing section of a line bundle $L$ and $\sigma'$ is a non-vanishing section of $L^*\otimes\pi^*E$. 
Equip $L$ with a Hermitian metric by setting $|s|_L^2:=|s\sigma'|^2_{\pi^*E}$ for any section $s$ of $L$. 
In particular, $|\pi^*\sigma|^2=|\sigma^0|^2_L$ on the union $X'$ of the components of $X$ where
$\pi^*\sigma$ is generically non-vanishing. By the Poincar\'e--Lelong formula thus
$$
dd^c\log\big|\pi^*\sigma|_{X'}\big|^2=dd^c\log|\sigma^0|^2_L=\div \sigma^0 - c_1(L).
$$
Since $\pi$ is  a biholomorphism generically we have $1=\pi_*1$ as currents. By \eqref{bara0} thus,
\begin{equation}\label{ghana}
(2\pi i)^{-1}\partial\log|\sigma|^2\w (dd^c\log|\sigma|^2)^{k-1}
=
\pi_*\big((2\pi i)^{-1}\partial\log|\sigma^0|_L^2\w (-c_1(L))^{k-1}\big)
\end{equation}
outside $Z$. Defining $\partial\log|\sigma^0|_L^2\w (-c_1(L))^{k-1}=0$ on $X\setminus X'$
the lemma follows since the right-hand side then is in $ASM(Y)$ and vanishes on irreducible components of $Y$
contained in $Z$.
\end{proof}

By Lemma~\ref{ASMPM} we can now define our first operation:
\begin{equation}\label{barra}
m_k^\sigma\colon\PM_Y\to\PM_Y,\quad\quad \tau\mapsto m^\sigma_k\w \tau, \quad k=1,2,\ldots. 
\end{equation}
In view of Lemma~\ref{ASMPM}, $\1_Z m^\sigma_k\w \tau=0$, and if $\tau$ has support in $Z$, then $m_k^\sigma\wedge\tau=0$. 
For any pseudomeromorphic $\tau$,
by the comment after Lemma~\ref{ASMPM}, we have 
\begin{equation}\label{lillamdef}
m_k^\sigma\wedge\tau=\lim_{\e\to 0} \chi(|\sigma|^2/\e) (2\pi i)^{-1}\partial\log|\sigma|^2\w (dd^c\log|\sigma|^2)^{k-1}\wedge\tau.
\end{equation}

\smallskip

Our second operation is the following:
\begin{equation*}
M_k^\sigma\colon\PM_Y\to\PM_Y, \quad
M_0^\sigma\wedge\tau = \1_{Z} \tau,
\quad  M^\sigma_k\w\tau=\1_Z\dbar (m_k^\sigma\wedge\tau), \quad k=1,2,\ldots.
\end{equation*}
If $\tau$ has support in $Z$, then $M_0^\sigma\wedge\tau=\tau$ and
$M_k^\sigma\wedge\tau=0$, $k=1,2,\ldots$.
One can check that for any pseudomeromorphic $\tau$,
\begin{equation}\label{Mdef}
M_0^\sigma\wedge\tau = \lim_{\e\to 0}(1-\chi(|\sigma|^2/\e))\tau, \quad
M_k^\sigma\wedge\tau = \lim_{\e\to 0}\dbar \chi(|\sigma|^2/\e)\wedge m_k^\sigma\wedge\tau,\quad k=1,2,\ldots.
\end{equation}
If $\alpha$ is a smooth form, then for instance by \eqref{Mdef},
\begin{equation}\label{prog1}
\alpha\w M^\sigma_k\w\tau=M^\sigma_k\w (\alpha\w\tau).
\end{equation}
Moreover, if $\tau=g_*\nu$ is pseudomeromorphic, where  $g\colon X\to Y$ is a proper holomorphic mapping and $\nu$ is \pmm on $X$, 
then in view of \eqref{bara0}, \eqref{Mdef}, and \eqref{lillamdef}
\begin{equation}\label{morr}
M_k^\sigma\wedge\tau = g_*(M^{g^*\sigma}_k\wedge\nu), \quad \quad m_k^\sigma\wedge\tau=g_*(m^{g^*\sigma}_k\wedge\nu).
\end{equation} 

We write $M_k^\sigma$ instead of $M_k^\sigma\w 1$. By the dimension principle, $M_k^\sigma=0$ for $k<\text{codim}\, Z$.

\begin{ex}\label{fredag}
Assume that $Z$ has pure codimension $\kappa$.
In view of \cite[Corollary~1.3]{aeswy1} we have $M_\kappa^\sigma=\sum_ja_j[Z_j]$, where $Z_j$ are the
irrecucible components of $Z$ and $a_j$ are positive integers. 
It follows that 
(the Lelong current of) any cycle  is pseudomeromorphic. 

If $Y$ is smooth and $\sigma$ generates the ideal sheaf $\J_Z$ of holomorphic functions vanishing on $Z$, 
then $M_\kappa^\sigma=[Z]$; cf.\ \cite[Corollary~1.3]{aswy}.
If $Y$ is not smooth the same holds if no component of $Z$ is contained in $Y_{sing}$. In fact, 
it holds on $Y_{reg}$ and therefore it holds on $Y$ by the dimension principle.
\end{ex}

\smallskip

Assume now that there is a regular sequence $f=(f_1,\ldots,f_\kappa)$ at $x\in Y$ such that $\{f=0\}=Z$; 
in particular then $\text{codim}\, Z_x=\kappa$.
We can considering $f$ as a section of the trivial rank-$\kappa$ vector bundle equipped with the trivial metric. 
Then 
$\dbar m_\kappa^f=0$ outside $Z$. By Lemma~\ref{ASMPM} we get that if $\tau$ is pseudomeromorphic, 
then $\1_{Y\setminus Z}\dbar(m_\kappa^f\wedge\tau)=-m_\kappa^f\w\dbar\tau$. It follows that
\begin{equation}\label{re3}
\dbar (m^f_\kappa\w\tau)=\1_Z\dbar(m^f_\kappa\w\tau) + \1_{Y\setminus Z}\dbar(m^f_\kappa\w\tau)
=M^f_\kappa\w\tau-m^f_\kappa\w \dbar\tau
\end{equation}
and, by applying $\dbar$ to \eqref{re3}, that 
\begin{equation}\label{re2}
\dbar (M^f_\kappa\w\tau)=M_\kappa^f\w\dbar\tau.
\end{equation}

\begin{ex}\label{BMex}
Let  $(\zeta,z)$ be coordinates in $\C^n\times \C^n$ and let $\eta=\zeta-z$. Let $p_1,p_2\colon\C^n\times \C^n\to \C^n$ be the projections
on the first and second factor, respectively, and let $i\colon \C^n\to\C^n\times \C^n$ be the diagonal embedding. 
We claim that if $\mu$ is pseudomeromorphic in 
$\mathcal{U}\subset\C^n$ then
\begin{equation}\label{extra}
M_n^\eta\wedge (1\otimes \mu) = i_*\mu
\end{equation}
in $\mathcal{U}\times \mathcal{U}$. If additionally $\mu$ has compact support, then
\begin{equation}\label{extra2}
\mu = \dbar \big((p_1)_* (m_n^\eta\wedge (1\otimes \mu))\big) + (p_1)_*(m_n^\eta\wedge(1\otimes \dbar\mu)).
\end{equation}

To see \eqref{extra}, let $[\Delta]:=i_*1$ be the diagonal. Then $[\Delta]\wedge(1\otimes \mu)$ is well-defined as a tensor product; in fact,
in the coordinates $(\eta,z)=(\zeta-z,z)$ it is just $\delta_0(\eta)\otimes \mu(z)$. It follows that
\begin{equation*}
[\Delta]\wedge(1\otimes \mu) = i_*\mu
\end{equation*}
in $\mathcal{U}\times \mathcal{U}$. Since by Example~\ref{fredag}
we have $M_n^\eta=[\Delta]$,
\eqref{extra} follows. To see \eqref{extra2}, notice that by \eqref{re3} we have
\begin{equation*}
\dbar(m_n^\eta\wedge (1\otimes\mu)) = 
M_n^\eta\wedge(1\otimes \mu) - m_n^\eta\wedge (1\otimes\dbar\mu).
\end{equation*}
If $\mu$ has compact support then we can apply $(p_1)_*$. By \eqref{extra}, and since $p_1\circ i=\text{id}$, thus \eqref{extra2} follows.
Notice that $m_n^\eta=\partial |\eta|^2\wedge (\dbar\partial |\eta|^2)^{n-1}/(2\pi i|\eta|^{2})^n$
is the Bochner-Martinelli kernel and that  \eqref{extra2} is the Bochner-Martinelli formula.
\end{ex}

\begin{remark}\label{altMdef}
Let $\sigma$ be a holomorphic section of $E\to Y$ as in the beginning of this section.
Assume that $\tau=g_*\beta$ is pseudomeromorphic, where $g\colon X\to Y$ is proper and $\beta$ is a product of components of 
Chern forms of various Hermitian vector bundles over $X$. Then $\tau$ is a \emph{generalized cycle}, a notion that was introduced in \cite{aeswy1}.
It follows from \cite[Section~5]{aeswy1} that 
$$
M_k^\sigma\wedge\tau=\1_Z (dd^c\log |\sigma|^2)^k\w\tau,
$$
where 
$$
(dd^c\log |\sigma|^2)^k\w\tau=\lim_{\e\to 0} (dd^c\log(|\sigma|^2+\e))^k\w\tau.
$$
\end{remark}

 \subsection{The inequality \eqref{codimolikhet}}\label{kodimolikhetsektion}
Let us recall why \eqref{codimolikhet} holds in a complex manifold $Y$.  By induction it is enough to verify it for
two analytic sets $A_1$ and $A_2$ in $Y$. Let $n=\text{dim}\, Y$. 
Notice that if $i\colon Y \to Y\times Y$ is the diagonal embedding and $\Delta=i_*Y$, then
$$
i_*(A_1\cap A_2)=(A_1\times A_2)\cap  \Delta.
$$
Locally in $Y\times Y$, $\Delta$ is defined by $n$ functions so that the right-hand side is obtained by
successive intersection by $n$ divisors (this is not true when $Y$ is singular so then
the argument breaks down). 
It is well-known that each such intersection can decrease
the dimension by at most one unit,
see, e.g.,  \cite[Theorem II~6.2]{Dem}. 
Hence $\text{dim}\, i_*(A_1\cap A_2)\geq \text{dim}\,A_1+\text{dim}\,A_2-n$ and so \eqref{codimolikhet} follows.  

From  \eqref{codimolikhet} we get the following useful observation.

\begin{lma}\label{gurka} 
Assume that $A_1,\ldots, A_r$ are germs of analytic sets of pure codimensions at a point in a smooth manifold $Y$. 
If the intersection $A_1\cap \cdots \cap A_r$ is proper, then each intersection 
$A_1\cap \cdots \cap A_\nu$, $\nu\le r$, is proper.
\end{lma}

\begin{proof}
Assume that $A_1,\ldots, A_r$ have codimensions $\kappa_1,\ldots \kappa_r$, respectively. By assumption,
$\codim (A_1\cap \cdots \cap A_r)=\kappa_1+\cdots +\kappa_r$. 
In view of \eqref{codimolikhet} thus
\begin{multline*}
\kappa_1+\cdots +\kappa_r= \codim\big((A_1\cap \cdots \cap A_\nu)\cap A_{\nu+1}\cap \cdots\cap A_r\big)\le \\
=\codim (A_1\cap \cdots \cap A_\nu)+\kappa_{\nu+1}+\cdots +\kappa_r
\le \kappa_1+\cdots \kappa_r.
\end{multline*}
All inequalities thus are equalities and we conclude that $\codim (A_1\cap \cdots \cap A_\nu)=\kappa_1+\cdots +\kappa_\nu$
as desired. 
\end{proof}

\section{A $\dbar$-potential approach to proper intersections}\label{frame}
Let $Y$ be a reduced analytic space of pure dimension $n$. We now describe our intrinsic approach
to a proper intersection product in $Y$ of cycles which locally admit good
potentials.

 \begin{df}
 Assume that $\mu$ is a cycle in $Y$ of pure codimension $\kappa\ge 1$ with support $Z$.
 We say that a \pmm current $u$ of bidegree
 $(\kappa,\kappa-1)$ in an open subset $\U\subset Y$  is a {\it good $\dbar$-potential} (or simply a \emph{good potential}) of $\mu$
 if $u$ is smooth in $\U\setminus Z$ and $\dbar u=\mu$.
 \end{df}
 
It follows from the dimension principle (see Section~\ref{PMsektion}) that $\1_Z u=0$.
If $\chi_\epsilon=\chi(|f|^2/\epsilon)$, where $Z(f)=Z$, therefore
$\chi_\epsilon u$ are smooth and $\chi_\epsilon u\to u$; cf.\ \eqref{rest0} and \eqref{reimond}. 
Since $\dbar u=0$ outside $Z$ thus $\mu_\epsilon 
=\dbar\chi_\e\w u$
are smooth, $\dbar$-closed, and $\mu_\epsilon\to \mu$.

\smallskip

Let $\mu_1$ and $\mu_2$ be (germs of) cycles of pure codimensions $\kappa_j\geq 1$ at a point $x\in Y$, let $Z_j=|\mu_j|$, and 
assume that $u_j$ are good potentials of $\mu_j$, $j=1,2$.

\begin{prop} \label{baker1}
The smooth form $u_2\w u_1$,  
a~priori defined in the Zariski open set $Z_2^c\cap Z_1^c$, has a unique
\pmm extension $T$ to $Y$ such that  $1_{Z_2\cup Z_1}T=0$.
\end{prop}

\begin{proof}
If $T$ and $T'$ are \pmm currents with the stated properties, then 
$$
T-T'=\1_{Z_2^c\cap Z_1^c}(T-T')+\1_{Z_2\cup Z_1}(T-T')=0
$$
since they coincide in $Z_2^c\cap Z_1^c$ 
and both vanish on the complement.

For the existence of the extension, 
let $i\colon Y\to Y\times Y$ be the diagonal embedding and let $\Delta=i_*Y$.
Since $\Delta$ is not contained in $(Y\times Y)_{sing}$ it follows from Example~\ref{fredag} that
$M_n^\eta=[\Delta]$, where $\eta$ is a tuple of holomorphic functions defining $\Delta$. 
By  \eqref{prog1}  we have
$$
M_n^\eta\wedge(u_2\otimes u_1)=(u_2\otimes u_1)\w M_n^\eta= (u_2\otimes u_1)\w [\Delta] = i_*i^*(u_2\otimes u_1)
$$
in $Z_2^c\times Z_1^c$. 
If $p\colon Y\times Y\to Y$ is the projection, e.g., on the first factor, then $p\circ i=\text{id}_Y$ and
it follows that 
$
p_*\big(M_n^\eta\w (u_2\otimes u_1)\big)
$
is equal to $u_2\w u_1$ in  $Z_2^c\cap Z_1^c$. 
Since $p$ is a simple projection, $p_*$ preserves pseudomorphicity and thus 
 $$
 T:=\1_{Z_2^c\cap Z_1^c} p_*\big(M_n^\eta\wedge (u_2\otimes u_1)\big)
 $$
 is the desired \pmm extension.
\end{proof}

We denote the extension $T$ by $u_2\w u_1$ 
as well. It is immediate from the proposition that
this product is anti-commutative since its restriction to 
$Z_2^c\cap Z_1^c$ 
is.

\smallskip

We will now define the product of $u_2$ by $\mu_1$. Notice that $u_2\w \mu_1$ is well-defined in $Z_2^c$. We claim that
$\mathcal{T}=-\1_{Z_2^c}\dbar (u_2\wedge u_1)$
is the unique \pmm extension of $u_2\w \mu_1$ to $Y$ such that $\1_{Z_2}\mathcal{T}=0$. The uniqueness of such an extension
is clear, and
since $u_2$ is $\dbar$-closed in $Z_2^c$,  $\mathcal{T}$ is indeed an extension, so the claim follows. We denote this extension by
$u_2\w \mu_1$.  That is,  
\begin{equation}\label{totem1}
u_2\w\mu_1=-\1_{Z_2^c}\dbar (u_2\wedge u_1).
\end{equation}
It follows from the dimension principle that if \eqref{barbados} below holds and $T$ is any pseudomeromorphic
current of pure bidegree such that $T=u_2\w\mu_1$ in $Z_2^c$
and $\text{supp}\,T\subset Z_1$, then $T=u_2\w\mu_1$ in $Y$. 

We now define
\begin{equation}\label{totem15}
\mu_2\w \mu_1:=\dbar(u_2\w \mu_1).
\end{equation}
It is clear that  $\mu_2\w \mu_1$ has support on $Z_2\cap Z_1$. However, without further assumptions it may depend on the choice of
potential $u_2$. 
If $f$ is a holomorphic tuple with zero set $Z_2$ we have 
\begin{equation}\label{smalben}
u_2\w\mu_1=\lim_{\e\to 0}\chi(|f|^2/\e)u_2\w\mu_1
\end{equation}
since $\1_{Z_2}u_2\w\mu_1=0$, cf.\ \eqref{rest0} and \eqref{reimond}. Since $\dbar(u_2\w\mu_1)=0$ in $Z_2^c$
it follows that 
$$\mu_2\w\mu_1=\lim_{\e\to 0}\dbar\chi(|f|^2/\e)\w u_2\w\mu_1.$$

\smallskip

\begin{df}\label{properdef}
We say that $\mu_2$ and $\mu_1$ intersect properly if 
\begin{equation}\label{barbados}
\codim (Z_2\cap  Z_1)\ge \codim Z_2+\codim Z_1.
\end{equation}
\end{df}

Since $\mu_2\w\mu_1$ has support in $Z_2\cap Z_1$ it follows from the dimension principle that $\mu_2\w\mu_1=0$ if 
$\codim (Z_2\cap  Z_1)> \codim Z_2+\codim Z_1$.

\begin{prop}\label{snoddas1}
Assume that  $Z_2$ and $Z_1$ intersect properly.
With the notation above,  
\begin{equation}\label{snorre}
\mu_2\w\mu_1=\mu_1\w\mu_2,
\end{equation}
and the product is independent of the choice of good potentials in the
definition.
\end{prop}

We have already noticed that $\mu_2\w \mu_1$ has support on $Z_2\cap Z_1$. 

\begin{proof}
We claim that
\begin{equation}\label{snorre1}
\dbar (u_2\w u_1)=  -u_2\w \mu_1+u_1\w \mu_2.
\end{equation}
In fact,
$\dbar(u_2\w u_1)=0$ in the Zariski open set $Z_1^c\cap Z_2^c$, and so  
$$
\1_{Z_1^c\cap Z_2^c}\dbar(u_2\w u_1)=0.
$$
Recall that $\kappa_j=\codim Z_j$.
The current 
$
T:=\1_{Z_1\cap Z_2}\dbar(u_2\w u_1)
$
has bidegree $(*,\kappa_2+\kappa_1-1)$, and 
by the assumption of proper intersection it has support on a set of
codimension $\ge \kappa_2+\kappa_1$. Thus  $T$ vanishes by the dimension principle.
In view of \eqref{totem1} and the equality
\begin{equation}\label{nonsens}
\1= \1_{Z_1^c}+\1_{Z_2^c} - \1_{Z_1^c\cap Z_2^c}+\1_{Z_1\cap Z_2}
\end{equation}
now \eqref{snorre1} follows.
If we apply $\dbar$ to \eqref{snorre1} we get \eqref{snorre}, cf.~\eqref{totem15}.
 
From the very definition of  $\mu_2\w\mu_1$ it is clear that it does not depend on the potential
$u_1$. In view of \eqref{snorre} it does not depend on $u_2$. 
 \end{proof}

The product $\mu_2\w\mu_1$ is $\mathbb{Q}$-bilinear in the following sense. If $\mu_2=a\mu'_2+b\mu''_2$, where 
$\mu'_2$ and $\mu''_2$ have good potentials and intersect $\mu_1$ properly and $a,b\in\mathbb{Q}$,
then 
\begin{equation}\label{blae}
a\mu'_2\w\mu_1 + b\mu''_2\w\mu_1=\mu_2\w\mu_1.
\end{equation}
Since $\mu_2\w\mu_1=\mu_1\w\mu_2$ the roles of $\mu_2$ and $\mu_1$ can be interchanged.
To see \eqref{blae}, let $u'_2$ and $u''_2$ be good potentials of $\mu'_2$ and $\mu''_2$, respectively, and let 
$\tilde u_2=au'_2+bu''_2$. Then $\dbar\tilde u_2=\mu_2$ but $\tilde u_2$ is not necessarily a good potential of $\mu_2$. 
However, the proof of Proposition~\ref{snoddas1} goes through with $u_2$ replaced by $\tilde u_2$ and $Z_2$ replaced by
$|\mu'_2|\cup |\mu''_2|$. It follows that
$$
\dbar(\tilde u_2\w\mu_1) = \mu_1\w\mu_2,
$$
from which we see that \eqref{blae} holds.

\begin{remark} 
With similar techniques as in the proof of Proposition~\ref{snoddas1} one can prove that $\mu_2\w\mu_1$ is $d$-closed.
However, unfortunately we cannot show, in general, that $\mu_2\w\mu_1$ is (the Lelong current of) a cycle.
\end{remark}

\begin{ex}\label{goodex}
Assume that $f=(f_1,\ldots,f_{\kappa_2})$ is a regular sequence at $x\in Y$ and let $\mu_2=M_{\kappa_2}^f$, cf.\ Example~\ref{fredag}.
In view of Section~\ref{oponPM}, $m_{\kappa_2}^f$ is a good potential of $\mu_2$. If $\mu_1$ is a cycle which has a local good potential, then
\begin{equation}\label{krak}
\mu_2\w\mu_1=M_{\kappa_2}^f\w\mu_1,
\end{equation}
where the right-hand side is defined in Section~\ref{oponPM}. To see this, recall that $m_{\kappa_2}^f\w\mu_1$, 
defined as in Section~\ref{oponPM}, is the unique \pmm extension to $Y$ of the natural product of $m_{\kappa_2}^f$ and $\mu_1$ in $Z_2^c$
such that $\1_{Z_2}m_{\kappa_2}^f\w\mu_1=0$. Therefore $m_{\kappa_2}^f\w\mu_1$ is the product of the good potential $m_{\kappa_2}^f$ and $\mu_1$.
Now \eqref{krak} follows by \eqref{totem15} and \eqref{re3}.
\end{ex}

For degree reasons, cycles in $Y$ of codimension $0$ cannot have good potentials. Such a cycle is a linear 
combination of the irreducible components of $Y$. If $\mu_1=aY$, $a\in\mathbb{Q}$, and $\mu_2$ is a cycle with a good potential $u_2$,
then we define $u_2\wedge \mu_1:=au_2$ and $\mu_2\wedge\mu_1:=a\mu_2$, which is consistent with \eqref{totem15}. If $\mu_1$ is any other 
cycle of codimension $0$, then we do not give any meaning to $\mu_2\w\mu_1$.

\begin{remark}
One can extend the approach in this section to more than two factors. If $u_1,\ldots,u_r$ are good potentials of cycles
$\mu_1,\ldots,\mu_r$ with $|\mu_j|=Z_j$, then $u_r\w\cdots\w u_1$, a~priori defined in $\cap_jZ_j^c$, has a unique pseudomeromorphic
extension to $Y$ whose restriction to $\cup_jZ_j$ vanishes. One can then recursively define
\begin{equation*}
u_r\w\cdots\w u_{\ell+1}\w\mu_{\ell}\w\cdots\w\mu_1:=\pm \1_{\cap_{\ell+1}^r Z_j^c} \dbar(u_r\w\cdots\w u_{\ell}\w\mu_{\ell-1}\w\cdots\w\mu_1)
\end{equation*}
for $\ell=1,2,\ldots$. It turns out that if all subsets of $\{Z_1,\ldots,Z_r\}$ intersect properly in the sense of Definition~\ref{properdef},
then $\mu_r\w\cdots\w\mu_1$ is independent of the ordering of the factors and choice of good potentials, has support in $\cap_1^rZ_j$,
and is $d$-closed.
Since this extension is not needed in this paper we omit the details. 
\end{remark}

\section{Proper intersections when $Y$ is smooth}\label{glatt}
We shall now use the approach in the previous section in the case when $Y$ is smooth, and see that we get
back the classical proper intersection product.

 \begin{lma}\label{apa}
If $Y$ is smooth, then any cycle $\mu$ in $Y$ admits, locally, a good $\dbar$-potential. 
\end{lma}

\begin{proof}
It is sufficient to show the lemma when $Y$ is an open subset of $\C^n$. Let
$\U\Subset Y$ be pseudoconvex and let $\rho$ be a smooth cutoff function in $Y$ that is $1$ in 
$\U$. Let $m_n^\eta$ be the Bochner-Martinelli kernel in $\C^n\times\C^n$, see Example~\ref{BMex}.
By \eqref{extra2} applied to $\rho \mu$ we get that
\begin{equation}\label{apfel}
\mu= \dbar u' + (p_1)_*\big(m_n^\eta \wedge(1\otimes\dbar\rho\w \mu)\big)
\end{equation}
in $\U$, where $u'=(p_1)_*(m_n^\eta \wedge(1\otimes\rho\mu))$. Since $p_1$ is a simple projection, $u'$ is pseudomeromorphic.
Notice that $u'$ and the last term in \eqref{apfel} are the convolutions of the Bochner-Martinelli form in $\C^n$ by
$\rho\mu$ and $\dbar\rho\w\mu$, respectively. Hence they are smooth where $\rho\mu$ and 
$\dbar\rho\w\mu$, respectively, are smooth. In $\U$ thus $u'$ is smooth outside $|\mu|$ and 
the last term in \eqref{apfel} is smooth. Moreover, the last term in \eqref{apfel} is clearly $\dbar$-closed in $\U$
and so it is $\dbar u''$ for some smooth form $u''$ there. We conclude that $u'+u''$ is a good potential of $\mu$ in 
$\U$.
\end{proof}

We can thus use the approach in Section~\ref{frame} and obtain a 
commutative current product 
$\mu_2\w \mu_1$ for any cycles $\mu_1,\mu_2$ that intersect properly, and it has support on $|\mu_2|\cap|\mu_1|$.

\begin{prop}\label{glas}
If  $\mu_1$ and $\mu_2$ are any cycles in $Y$ that intersect properly, with supports $Z_1$ and $Z_2$,
respectively, then $\mu_2\w\mu_1$ is (the Lelong current of) a cycle
$\mu_2\cdot \mu_1$ with support on $Z_1\cap Z_2$. Moreover,
if $i\colon Y\to Y\times Y$ is the diagonal embedding and $\Delta=i_*Y$, then
\begin{equation}\label{diaformel}
i_*(\mu_2\w\mu_1)=i_*(\mu_2\cdot\mu_1)=[\Delta]\cdot (\mu_2\otimes \mu_1). 
\end{equation}
\end{prop}

\begin{proof}
The proposition is local so we may assume that $Y$ is an open subset of $\C^n$.
In view of \eqref{blae} and Lemma~\ref{apa} we can also assume that $\mu_j=|\mu_j|=Z_j$.
Let $u_2$ be a good potential of $\mu_2$. 
We first claim that 
\begin{equation}\label{pam0}
u_2\w \mu_1=p_*\big(M_n^\eta \w (u_2\otimes \mu_1)\big),
\end{equation}
where $p\colon Y\times Y\to Y$ is the projection on the first factor.
Notice that $u_2\otimes 1$ is smooth in 
$Z_2^c\times Y$.   Therefore, by \eqref{prog1}, \eqref{extra}, and \eqref{bara0}, in $Z_2^c\times Y$ we have
\begin{equation*}
M_n^\eta \w (u_2\otimes \mu_1) = (u_2\otimes 1)\wedge M_n^\eta\wedge (1\otimes \mu_1)=
(u_2\otimes 1)\wedge i_*\mu_1 = i_*(u_2\wedge \mu_1).
\end{equation*}
Since $p\circ i=\text{id}_Y$, \eqref{pam0} holds in $Z_2^c$.
Now $M_n^\eta \w (u_2\otimes \mu_1)$ has support in $\Delta\cap (Y\times Z_1)$. Thus,
its restriction to $Z_2\times Y$ has support in $\Delta\cap (Z_2\times Z_1)$ and hence it vanishes 
in view of  the dimension principle. By \eqref{bara1} thus,
\begin{equation*}
\1_{Z_2}p_*\big(M_n^\eta \w (u_2\otimes \mu_1)\big) =
p_*\big(\1_{Z_2\times Y}M_n^\eta \w (u_2\otimes \mu_1\big)=0.
\end{equation*}
Since also $\1_{Z_2}u_2\w \mu_1=0$, cf.\ \eqref{totem1}, it follows that \eqref{pam0} holds in $Y$. 

Applying $\dbar$ to \eqref{pam0} and using \eqref{re2} we get
\begin{equation}\label{pam10}
\mu_2\wedge\mu_1=p_*\big(M_n^\eta \w (\mu_2\otimes \mu_1)\big).
\end{equation}
To see that this is a cycle, notice that if $\iota\colon \mu_2\times\mu_1\to Y\times Y$ is the inclusion, then
by \eqref{morr},
$
M_n^\eta \w (\mu_2\otimes \mu_1) = \iota_*M_n^{\iota^*\eta}.
$
By King's formula, cf.\ \cite{aswy}, it follows that $M_n^{\iota^*\eta}$, 
and hence  
$M_n^\eta \w (\mu_2\otimes \mu_1)$, are cycles.  
Since $M_n^\eta \w (\mu_2\otimes \mu_1)$ has support in $\Delta\simeq Y$ 
there is a cycle $\mu$ in $Y$ such that 
\begin{equation}\label{tutti}
i_*\mu=M_n^\eta \w (\mu_2\otimes \mu_1).
\end{equation}
Thus, by \eqref{pam10},
$\mu_2\wedge\mu_1=p_*i_*\mu=\mu$
is a cycle. 

For the last statement, notice that $m_n^\eta$ is a good potential
of $[\Delta]$ in view of Example~\ref{goodex}. Thus, cf.\ Example~\ref{goodex},
\begin{equation}\label{arm}
M_n^\eta \w (\mu_2\otimes \mu_1)=[\Delta]\w(\mu_2\otimes\mu_1).
\end{equation}
Since this is a cycle we write the right-hand side as $[\Delta]\cdot(\mu_2\otimes\mu_1)$.
Now \eqref{diaformel} follows from \eqref{pam10}, \eqref{tutti}, and \eqref{arm}.
 \end{proof}

\begin{remark}
If $Y$ is singular, then the diagonal $\Delta\subset Y\times Y$ is not a regular embedding, i.e., defined by a locally complete intersection.
The proof of Proposition~\ref{glas} then breaks down because if $\eta$ (locally) defines $\Delta$, then 
$m_n^\eta$ is not $\dbar$-closed outside $\Delta$ and thus not a good potential of $\Delta$; cf.\ Example~\ref{fruktsallad} below.
\end{remark}

\subsection{Comparison to Chirka's approach}\label{Chirkaprod}
In \cite{Ch} the product 
of two properly intersecting cycles $\mu_1$ and $\mu_2$ can be obtained as follows,
see the theorem in \cite[Ch.\ 3, \S 16.2, p.\ 212]{Ch}. We denote it here by $\mu_2\curlywedge\mu_1$ to distinguish it from our product. 
If $\mu_2^{\epsilon}$ is a 
regularization of $\mu_2$ obtained by any standard approximate identity, then
\begin{equation}\label{chirkadef}
\lim_{\e\to 0}\mu_2^\epsilon\wedge \mu_1 =\mu_2\curlywedge \mu_1.
\end{equation}
\begin{prop}\label{chirkaprop}
If $\mu_1$ and $\mu_2$ intersect properly in the manifold $Y$, then $\mu_2\curlywedge\mu_1 = \mu_2\wedge\mu_1$.
\end{prop}

\begin{proof}
The statement can be checked locally so we can assume that $Y$ is an open subset of $\C^n$. 
Let $\phi(\zeta)=\dbar\chi(|\zeta|^2)\wedge m_n^\zeta$, where $\chi'\geq 0$, and let $\chi^\epsilon(\zeta)=\chi(|\zeta|^2/\epsilon)$. Then 
$\phi(\zeta)$ is a positive $(n,n)$-form and
\begin{equation}\label{prome}
\phi_\e(\zeta):=\phi(\zeta/\sqrt{\e})=\dbar\chi^\e(\zeta)\wedge m_n^\zeta
\end{equation}
is an approximate identity (considered as an $(n,n)$-form) in $\C^n$. If $p\colon Y\times Y\to Y$ is the projection on the second factor and
$\eta=\zeta-z$, then 
\begin{equation}\label{nad}
\mu_2^\e := p_*(\phi_\e(\eta)\wedge (\mu_2\otimes 1))
\end{equation}
is the convolution of $\mu_2$ and $\phi_\e$. By \eqref{chirkadef} thus $\lim_{\epsilon\to 0}\mu_2^\epsilon\wedge \mu_1= \mu_2\curlywedge \mu_1$.
The proposition now follows from the next lemma.
\end{proof}

\begin{lma}
We have $\lim_{\e\to 0}\mu_2^\epsilon\wedge \mu_1 = \mu_2\wedge \mu_1$.
\end{lma}

\begin{proof}
This can be checked locally so let $\U\Subset Y$ be pseudoconvex and let $\rho$ be a smooth cutoff function in $Y$ that
is $1$ in a neighborhood of $\overline\U$. 
In view of \eqref{nad} and \eqref{prome}, if $\e>0$ is sufficiently small, then in $\U$ we have
\begin{eqnarray*}
\mu_2^\e &=& p_*\big(\phi_\e(\eta)\wedge (\rho\mu_2\otimes 1)\big) \\
&=&
\dbar p_*\big(\chi^\e(\eta) m_n^\eta\wedge (\rho\mu_2\otimes 1)\big) + p_*\big(\chi^\e(\eta) m_n^\eta\wedge (\dbar\rho\wedge\mu_2\otimes 1)\big).
\end{eqnarray*}
The last term on the right-hand side is independent of $\e>0$ in $\U$ if $\e$ is sufficiently small $\e$. 
As in the proof of Lemma~\ref{apa} it follows that 
$$
\mu_2^\e=\dbar(u'_\e + u''),
$$
where $u'_\e=p_*\big(\chi^\e(\eta) m_n^\eta\wedge (\rho\mu_2\otimes 1)\big)$ and $u''$ is smooth in $\U$.
By \eqref{lillamdef},
$$
\lim_{\e\to 0}\chi^\e(\eta) m_n^\eta\wedge (\rho\mu_2\otimes 1)=m_n^\eta\wedge (\rho\mu_2\otimes 1)
$$
is pseudomeromorphic, and since $p$ is a simple projection,
$$
u':= \lim_{\e\to 0}u'_\e=p_*\big(m_n^\eta\wedge (\rho\mu_2\otimes 1)\big)
$$
is pseudomeromorphic. Moreover, $u'$ is smooth outside $|\mu_2|$, and we notice that 
the convergence $u'_\e\to u'$ is locally uniform in $\U\setminus |\mu_2|$.
Thus $u:=u'+u''$ is a good potential of $\mu_2$ in $\U$.

We claim that $u\w \mu_1=\lim_{\e\to 0}(u'_\e+u'')\w \mu_1$, where the left-hand side is the product in Section~\ref{frame}.
Taking the claim for granted the lemma follows since 
$$
\mu_2^\e\w \mu_1=\dbar(u'_\e+u'')\w \mu_1=\dbar\big((u'_\e+u'')\w \mu_1\big)\to \dbar (u\w \mu_1)=\mu_2\w \mu_1.
$$
To show the claim, notice first that by standard distribution theory,
\begin{equation*}
u'_\e\wedge \mu_1 = p_*\big(\chi^\e(\eta) m_n^\eta\wedge(\rho\mu_2\otimes\mu_1)\big).
\end{equation*}
Then $\lim_{\e\to 0} u'_\e\wedge \mu_1=p_*\big(m_n^\eta\wedge(\rho\mu_2\otimes \mu_1)\big)$ is pseudomeromorphic,
has support in $|\mu_1|$, and is equal to
$u'\w \mu_1$ in $\U\setminus |\mu_2|$ since $u'_\e\to u'$ locally uniformly there. 
Hence $T:=\lim_{\e\to 0}(u'_\e+u'')\w \mu_1$ is pseudomeromorphic, has support in $|\mu_1|$, and $T=u\w \mu_1$ in $\U\setminus |\mu_2|$. 
Since $\mu_2$ and $\mu_1$ intersect properly thus the claim follows by the dimension principle; cf.\ the comment after \eqref{totem1}.
\end{proof}

\subsection{Comparison to the algebraic definition in \cite{Fult}}\label{Fultprod}
The intersection product in \cite{Fult} of two cycles $\mu_1$ and $\mu_2$, that we here denote by $\mu_2\bullet \mu_1$,
is defined by the intersection of the diagonal
$\Delta\subset Y\times Y$ and the product cycle $\mu_2\times\mu_1$. In view of Proposition~\ref{glas},
to see that the intersection product in \cite{Fult} in the case of proper intersections coincides with our product it thus suffices to see that 
if $\mu\subset Y$ is an irreducible subvariety and $A\subset Y$ is a submanifold  intersecting $\mu$ properly, then $A\bullet \mu = [A]\wedge \mu$.

Assume that $\text{codim}\, A=\kappa$ and $\text{dim}\,\mu=d$. 
In general, if $A$ and $\mu$ do not necessarily intersect properly, then $A\bullet \mu$ is a Chow class 
of dimension $d-\kappa$ in the set-theoretic intersection $A\cap \mu$.
If $A$ and $\mu$ intersect properly, then $\text{dim}\, A\cap\mu=d-\kappa$ and this Chow class is the zeroth Segre class
of the subspace $A\cap\mu$ of $\mu$, $s_0(A\cap\mu,\mu)$, which is a cycle with support $A\cap\mu$. 
If $\iota\colon \mu\to Y$ and $j\colon A\cap\mu\to\mu$
are the inclusions, then as cycles in $Y$ we have
\begin{equation*}
A\bullet \mu = \iota_*j_* s_0(A\cap\mu,\mu).
\end{equation*}

Let $f=(f_1,\ldots,f_\kappa)$ be a tuple defining $A$ (locally) in $Y$.
By Example~\ref{fredag} and \eqref{re3} then $m_\kappa^f$ is a good potential of $[A]$.
In view of \eqref{re3} and \eqref{morr} thus
\begin{equation*}
[A]\wedge\mu=\dbar(m_\kappa^f\wedge \mu) = M_\kappa^f\wedge\mu = \iota_*M_\kappa^{\iota^*f}.
\end{equation*} 
It follows from \cite[Proposition~1.5]{aeswy1} that $M_\kappa^{\iota^*f}=j_*s_0(A\cap\mu,\mu)$ and hence
$[A]\wedge\mu = A\bullet \mu$.

\section{Proper intersection of nice cycles}\label{nicecycles}

To begin with, assume that $Y$ is a complex manifold and let $i\colon Y\to Y'$ be a local embedding into a complex manifold $Y'$. 
Let $\mu_1,\ldots,\mu_r$ be germs of effective cycles in $Y$ at a point $x$ and let $\mu'_j$ be effective cycles in $Y'$ intersecting
$i_*Y$ properly and such that 
$
i_*\mu_j = \mu'_j\cdot_{Y'} i_*Y.
$
For instance one can take $\mu'_j$ as follows. Let $(x,y)$ be coordinates in $Y'$ such that $i_*Y=\{y=0\}$. 
Possibly after shrinking $Y'$ and $Y$, we have $Y'=i_*Y \times U$ for some neighborhood $U$ of $0$ in some $\mathbb{C}^d$ and we can take 
$\mu'_j=i_*\mu_j\times U$.
For any choice of such $\mu'_j$ we have that $\mu'_1,\ldots,\mu'_r,i_*Y$ intersect properly in $Y'$ if and only if 
$\mu_1,\ldots,\mu_r$ intersect properly in $Y$, cf.\ Lemma~\ref{ryydiger} below. In this case,
\begin{equation}\label{rydiger}
i_*(\mu_1\cdot_Y\cdots\cdot_Y\mu_r) = \mu'_1\cdot_{Y'}\cdots\cdot_{Y'}\mu'_r\cdot_{Y'} i_*Y.
\end{equation}
We can thus relate proper intersections in $Y$ to proper intersections in a larger smooth ambient space.
This is the idea for defining proper intersections when $Y$ is singular and motivates the definition of nice cycles.

\begin{df}\label{nicedef}
Let $Y$ be a reduced analytic space of pure dimension $n$. A germ of an effective cycle $\mu$ at  $x\in Y$ is nice
if there is a local embedding $i\colon Y\to Y'$, where $Y'$ is smooth, and an effective cycle $\mu'$ in $Y'$ 
intersecting $i_*Y$ properly,  such that
\begin{equation}\label{kopp}
i_*\mu = \mu'\cdot_{Y'} i_*Y.
\end{equation}
\end{df}

We say that $\mu'$ is a \emph{representative} of $\mu$ in $Y'$. Since the intersection $\mu'\cdot_{Y'} i_*Y$ is proper it follows that 
\begin{equation}\label{pont100}
\codim_{Y} \mu=\codim_{Y'}\mu'.
\end{equation}
Since $\mu$ and $\mu'$ are effective we have that 
\begin{equation}\label{urk}
|i_*\mu|=|\mu'|\cap i_*Y.
\end{equation}
Moreover, 
if $\mu=\mu_1+\mu_2$, where $\mu_j$ are effective, then $|\mu|=|\mu_1|\cup |\mu_2|$.  
Neither this last statement nor \eqref{urk} holds if the assumptions on effectivity are dropped, not even if $Y$ is smooth.

From the definition we see that a $\mathbb{Q}_+$-linear combination of nice cycles is nice.

\begin{ex}\label{blyerts}
The germ of $Y$ at any point $x\in Y$ is nice since there is a local embedding $i\colon Y\to Y'$ of a neighborhood of $x$ into some open
$Y'\subset\mathbb{C}^N$ and $i_*Y=Y'\cdot_{Y'}i_*Y$. 
Since $aY'$, $a\in\mathbb{Q}_+$, are the only effective cycles in $Y'$ of codimension $0$, the only nice cycles in $Y$ at $x$ of codimension $0$ are 
$aY$. In particular, if $Y$ has several irreducible components at $x$, then neither of these are nice.
\end{ex}

\begin{ex}\label{spetsigt}
Let $Y=\{z^2=w^3\}\subset\C^2$ and let $p=0$. Then 
$Y\cdot_{\C^2}\{w=0\}=2p$ and $Y\cdot_{\C^2}\{z=0\}=3p$. Considering $p$ as a cycle in $Y$ thus $(1/2)\{w=0\}$ and $(1/3)\{z=0\}$ are 
representatives of $p$ in $\C^2$.
\end{ex}

\begin{lma}\label{ryydiger}
If $\mu_1,\ldots, \mu_r$ are germs of nice cycles at $x\in Y$ of pure codimensions $\kappa_1,\ldots,\kappa_r$,
respectively, 
then 
\begin{equation}\label{codimolikhetnice}
\codim\big(|\mu_1|\cap\cdots\cap|\mu_r|\big) \leq \kappa_1+\cdots +\kappa_r.
\end{equation}
Let $\mu'_j$ be representatives of $\mu_j$ in the same $Y'$. Then equality holds in \eqref{codimolikhetnice}
if and only if $\mu'_1,\ldots,\mu'_r,i_*Y$ intersect properly.
\end{lma}

\begin{proof}
Notice first that if $A\subset Y$, then
\begin{equation}\label{pont}
\codim A+\codim_{Y'} i_*Y= \codim_{Y'} i_*A. 
\end{equation}
In view of \eqref{urk}, $i_* \big(|\mu_1|\cap\cdots\cap|\mu_r|\big)= |\mu_1'|\cap\cdots\cap|\mu_r'|\cap i_*Y$. By
\eqref{codimolikhet} in $Y'$ and \eqref{pont} thus 
\begin{multline*}
\codim \big(|\mu_1|\cap\cdots\cap|\mu_r|\big)+\codim_{Y'} i_*Y
=\codim_{Y'}i_* \big(|\mu_1|\cap\cdots\cap|\mu_r|\big)=\\
\codim_{Y'}(|\mu_1'|\cap\cdots\cap|\mu_r'|\cap i_*Y)\le
\codim_{Y'}|\mu_1'|+\cdots + \codim_{Y'}|\mu_1'|+\codim_{Y'} i_*Y.
\end{multline*}
In view of \eqref{pont} and \eqref{pont100} both  \eqref{codimolikhetnice} and
the last statement of the lemma follow.
\end{proof}

\begin{lma}\label{snartsen}
Let $\mu$ be a germ of a nice cycle in $Y$ at $x$ and let $j\colon Y\to Y'_m$ be a minimal embedding of a neighborhood of $x$.
Then there is a representative of $\mu$ in $Y'_m$.
\end{lma} 

\begin{proof}
By definition there is an embedding $i\colon Y\to Y'$ and a representative $\mu'$ of $\mu$ in $Y'$.
Since $j$ is a minimal embedding there is an embedding $\iota\colon Y'_m\to Y'$ such that the composition
\begin{equation}\label{hastrusk}
Y\stackrel{j}{\longrightarrow} Y'_m\stackrel{\iota}{\longrightarrow} Y'
\end{equation}
is the embedding $i$.
By choosing suitable local coordinates in $Y'$, and possibly shrinking $Y'$ and $Y'_m$, we can assume that $Y'=\iota_*Y'_m\times U$
for some open $U$ in some $\C^d$. 

We claim that $\mu'$, $i_*Y\times U$, and $\iota_*Y'_m$ intersect properly in $Y'$. 
This follows since by definition $\mu'$ and $i_*Y$ intersect properly in $Y'$, and since we clearly have the proper intersection
\begin{equation}\label{tomatsallad}
(i_*Y\times U) \cdot_{Y'} \iota_*Y'_m = i_*Y.
\end{equation}
By Lemma~\ref{gurka} thus all pairs of $\mu'$, $i_*Y\times U$, $\iota_*Y'_m$ intersect properly.
Let 
$$
\mu'_m:=\mu'\cdot_{Y'} \iota_*Y'_m.
$$
Then $\mu'$ is a representative of $\mu'_m$, and in view of \eqref{tomatsallad}, $i_*Y\times U$ is a representative of $j_*Y$ in $Y'$.
By Lemma~\ref{ryydiger} thus $\mu'_m$ and $j_*Y$ intersect properly in $Y'_m$.
From \eqref{rydiger} and \eqref{tomatsallad} we now get
$$
\iota_*(\mu'_m\cdot_{Y'_m} j_*Y)=\mu'\cdot_{Y'} (i_*Y\times U)\cdot_{Y'} \iota_*Y'_m = \mu'\cdot_{Y'} i_*Y = i_*\mu=\iota_*j_*\mu.
$$
Since $\iota_*$ is injective thus $j_*\mu=\mu'_m\cdot j_*Y$, i.e., $\mu'_m$ is a representative of $\mu$ in $Y'_m$.
 \end{proof}

We can thus assume that $i\colon Y\to Y'$ is a minimal
embedding in Definition~\ref{nicedef}.  By uniqueness of minimal embeddings
it follows that if $\mu_1, \ldots, \mu_r$ are nice cycles in $Y$ at $x$, then all of them have
representatives $\mu_j'$ in the same (minimal) $Y'$.
Moreover, 
if $\mu$ is a nice cycle at $x$ and $i\colon Y\to Y'$ is any local embedding, then there is a representative $\mu'$ of $\mu$
in $Y'$. Indeed, using the notation of the above proof, if $\nu'$ is a representative of $\mu$ in $Y'_m$ we can take $\mu'=\iota_*\nu'\times U$.

We say that germs $\mu_1,\ldots, \mu_r$ of nice cycles at $x\in Y$ intersect properly
if equality holds in \eqref{codimolikhetnice}. 
It follows that if $\mu_1,\ldots, \mu_r$ intersect properly, then
each subset of them do as well, cf.~(the proof of) Lemma~\ref{gurka}.
Notice that if $i\colon Y\to Y'$ is a local embedding and $\mu'_j$ are representatives of $\mu_j$ in $Y'$, then by Lemma~\ref{ryydiger},
$\mu_j$ intersect properly
if and only if $\mu'_1,\ldots,\mu'_r, i_*Y$ intersect properly.

\begin{prop}\label{boman}
Assume that $\mu_1,\ldots, \mu_r$ are germs of nice cycles at $x\in Y$ that intersect properly.
Then there is a unique germ of a nice cycle $\mu_1\cdots\mu_r$ at $x$ such that if $i\colon Y\to Y'$ is
a local embedding and $\mu'_j$ are representatives of $\mu_j$ in $Y'$, then 
\begin{equation}\label{progsnitt100}
 i_*(\mu_1\cdots \mu_r)=\mu_1' \cdot_{Y'}\cdots \mu_r'\cdot_{Y'} i_*Y.
\end{equation}
\end{prop}

\begin{proof}
The uniqueness is clear. To show existence we take \eqref{progsnitt100} for one fixed local embedding $i$
as the definition of $\mu_1\cdots \mu_r$ 
and show that it is independent of representatives $\mu'_j$ and $i$.

If we have other representatives $\mu_j''$ in $Y'$ for $\mu_j$, then 
$\mu''_j\cdot_{Y'}i_*Y=\mu'_j\cdot_{Y'}i_*Y$ and hence, by commutativity
of proper intersections in a smooth space, 
$\mu_1'' \cdot_{Y'}\cdots \mu_r''\cdot_{Y'} i_*Y=\mu_1' \cdot_{Y'}\cdots \mu_r'\cdot_{Y'} i_*Y$.

Let $j\colon Y\to Y'_m$ be a minimal embedding and factorize $i$ as in \eqref{hastrusk}. 
As in the proof of Lemma~\ref{snartsen} we get representatives $\mu'_{m,j}$ of $\mu_j$ in $Y'_m$.
With the notation in that proof we have
$\iota_*(\mu'_{m,1}\cdots\mu'_{m,r}\cdot j_*Y)=\mu'_1\cdots\mu'_r\cdot i_*Y=i_*(\mu_1\cdots\mu_r)$.
Since $i_*=\iota_* j_*$ thus $\mu'_{m,1}\cdots\mu'_{m,r}\cdot j_*Y=j_*(\mu_1\cdots\mu_r)$.
By uniqueness of minimal embeddings and the independence of representatives it follows that $\mu_1\cdots\mu_r$ 
is independent of the embedding.

Clearly, $\mu_1' \cdot_{Y'}\cdots \mu_r'$ is a representative of $\mu_1\cdots \mu_r$ in $Y'$ and so
$\mu_1\cdots \mu_r$ is nice.
\end{proof}

In view of Lemma~\ref{snartsen} and Proposition~\ref{boman} the following definition makes sense.

\begin{df}\label{snittdefny}
The proper intersection product of properly intersecting germs of nice cycles 
$\mu_1,\ldots,\mu_r$ in $Y$ at $x$ is the nice cycle $\mu_1\cdots\mu_r$ such that \eqref{progsnitt100} holds.
\end{df}

The product $\mu_1\cdots \mu_r$ is commutative since the proper intersection product in $Y'$ is. 
Moreover, $\mu_1\cdots \mu_r$ is $\mathbb{Q}_+$-linear in each factor in the sense that
if, for some $j$, $\mu_j=a\nu_1+b\nu_2$, where $a,b\in\mathbb{Q}_+$ and $\nu_1$ and $\nu_2$ are nice, then
\begin{equation}\label{skorsten}
\mu_1\cdots \mu_r = a\mu_1\cdots \nu_1\cdots \mu_r + b\mu_1\cdots \nu_2\cdots \mu_r.
\end{equation}
Notice that both $\nu_1$ and $\nu_2$ intersect $\mu_1,\ldots,\mu_{j-1},\mu_{j+1},\ldots,\mu_r$ properly
since $|\mu_j|=|\nu_1|\cup |\nu_2|$.

\begin{ex}\label{punkt}
Assume that $p$ is a point in $Y$.  By the local parametrization theorem, $Y$ is locally embedded as a branched
cover in a \nbh of $0$ in $\C^N=\C^n\times \C^{N-n}$.  It follows that $p\in Y$ is the 
proper intersection of $Y$ and $z_{1}=\cdots =z_n=0$ in $\C^N$. Hence $\{ p\}$ is
a nice cycle in $Y$. 
 \end{ex}

\begin{ex} \label{potatis}
Let $Y=\{x_1x_2+x_3x_4=0\}$ in $\C^4$.  Then the analytic sets
$\mu_1=\{x_1=x_3=0\}$ and $\mu_2=\{x_2=x_4=0\}$ both have codimension $1$ in $Y$ but
their set-theoretic intersection is just the point $0$, which has codimension $3$. In view of 
Lemma~\ref{codimolikhetnice}
not both of them, and by symmetry thus none of them, is nice.  
\end{ex}

Here is an example of an irreducible $Y$ and a nice cycle of positive codimension whose irreducible
components are not nice, cf.\ Example~\ref{blyerts}.

\begin{ex}
Let $Y$ be as in Example~\ref{potatis}; it is irreducible. The subvariety $\{ x\in Y;\ x_1=0\}$ is certainly nice
and it has the two
irreducible components $\mu_1=\{x_1=0,\ x_3=0\}$ and $\{x_1=0,\ x_4=0\}$. 
It follows from Example~\ref{potatis} that none of them is nice.
\end{ex}

\begin{prop}\label{trunk}
Assume that $i\colon Y\to Y'$ is a local embedding and $Y'$ is smooth. If $\mu'$ is a representative of 
the nice cycle $\mu$ in $Y'$
and $u'$ is a good potential of $\mu'$, then
there is a unique good potential $u$ of $\mu$ such that 
$u=i^*u'$ in $Y\setminus |\mu|$.
\end{prop}

\begin{proof}
The uniqueness is clear in view of the dimension principle.
By Lemma~\ref{apa}, $i_*Y$ has a good potential in $Y'$ and from Section~\ref{frame} it follows that
$u'\w i_*Y$, a~priori defined in $Y'\setminus |\mu'|$, has a unique \pmm extension to
$Y'$ such that $\1_{|\mu'|} (u'\w i_*Y)=0$. 
If $\chi_\epsilon=\chi(|f|^2/\e)$ and $\{f=0\}=|\mu'|$, then by \eqref{smalben} we have 
\begin{equation}\label{turtur}
\chi_\epsilon u'\w i_*Y\to u'\w i_*Y.
\end{equation}

Let $\xi$ be any smooth form in $Y'$ such that $i^*\xi=0$. Then 
clearly $\xi\w \chi_\epsilon u'\w i_*Y=0$ for $\e>0$. By \eqref{turtur} therefore
$\xi\w u'\w i_*Y=0$. This means that there is a (unique) current $u$ in
$Y$ such that  $i_* u=u'\w i_*Y$; cf.\ Section~\ref{prel}.  We claim that
\begin{equation}\label{turtur0}
\1_{i_*Y_{sing}} (u'\w i_*Y)=0.
\end{equation}
Taking the claim for granted for the moment we can complete the proof. Since $u'\w i_*Y$ is \pmm and in addition
\eqref{turtur0} holds, it follows
 from  \cite[Theorem~1.1]{litennot} that $u$ is \pmm in $Y$. 
 %
Moreover, in view of \eqref{totem1} and Section~\ref{Chirkaprod} or Section~\ref{Fultprod},
$$
i_* \dbar u=\dbar i_*u=\dbar (u'\w i_*Y)=\mu'\w i_*Y=\mu'\cdot i_*Y =i_*\mu
$$
so that $\dbar u=\mu$. In $Y'\setminus |\mu'|$, where $u'$ is smooth, we have 
$i_* i^* u'=u'\w i_*Y$.
Thus $u=i^* u'$ in $Y\setminus |\mu|$ and is smooth there.

To show the claim, notice first that $\1_{|\mu'|}\1_{i_*Y_{sing}}(u'\w i_*Y)=0$ in view of \eqref{rest1} since $\1_{|\mu'|} (u'\w i_*Y)=0$.
Moreover, $u$ is smooth in $Y'\setminus |\mu'|$ and so, by \eqref{rest2}, $\1_{i_*Y_{sing}}(u'\w i_*Y)=u'\w \1_{i_*Y_{sing}}i_*Y=0$ there.
It follows that $\1_{Y'\setminus |\mu'|}\1_{i_*Y_{sing}}(u'\w i_*Y)=0$. Hence,
$$
\1_{i_*Y_{sing}}(u'\w i_*Y) = \1_{|\mu'|}\1_{i_*Y_{sing}}(u'\w i_*Y) + \1_{Y'\setminus |\mu'|}\1_{i_*Y_{sing}}(u'\w i_*Y)=0
$$
concluding the proof.
\end{proof}

From Lemma~\ref{apa} we know that $\mu'$ has a good potential and hence we have

 \begin{cor}\label{alvinn}
If $\mu$ is a nice cycle in $Y$, then $\mu$ locally has a good potential. 
\end{cor}

We will now prove Theorem~\ref{nymain1}, which says that the extrinsic Definition~\ref{snittdefny} of the proper intersection
product coincides with intrinsic $\dbar$-potential-theoretic definition \eqref{totem15}.

\begin{proof}[Proof of Theorem~\ref{nymain1}]
Statement (i) follows from Corollary~\ref{alvinn}. Statement (ii) follows from the paragraph after the proof of 
Proposition~\ref{baker1}, cf.\ \eqref{totem1}. To prove (iii), let
$i\colon Y\to Y'$ be a local embedding, $\mu'_j$ representatives of $\mu_j$ in $Y'$, $u'_2$ a good potential of $\mu'_2$,
and $u_2$ the good potential
of $\mu_2$ such that $i^*u'_2=u_2$ outside $|\mu_2|$; cf.\ Proposition~\ref{trunk}. 
By Lemma~\ref{ryydiger}, $\mu'_1$, $\mu'_2$, and $i_*Y$ intersect properly, and by 
Lemma~\ref{gurka} any subset intersect properly too. We claim that
\begin{equation}\label{soligt}
i_*(u_2\w \mu_1) = u'_2\w (\mu'_1\cdot i_*Y).
\end{equation}
Indeed, it holds in $Y'\setminus |\mu'_2|$ since there
$$
u'_2\w (\mu'_1\cdot i_*Y) = u'_2\w i_*\mu_1 = i_*(i^*u'_2\w \mu_1) = i_*(u_2\w \mu_1).
$$
Since $\mu'_2$ and $\mu'_1\cdot i_*Y$ intersect properly and the left-hand side
of \eqref{soligt} has support in $|\mu'_1\cdot i_*Y|$ it follows from the dimension principle that 
\eqref{soligt} holds in $Y'$; cf.\ the comment after \eqref{totem1}.

In view of \eqref{totem15} and Section~\ref{Chirkaprod} or Section~\ref{Fultprod} we now get
\begin{eqnarray*}
i_*(\mu_2\w\mu_1) &=& i_*\dbar(u_2\w\mu_1)=\dbar i_*(u_2\w\mu_1) = \dbar (u'_2\w (\mu'_1\cdot i_*Y))
=\mu'_2\w (\mu'_1\cdot i_*Y)\\
&=&\mu'_2\cdot \mu'_1\cdot i_*Y=i_*(\mu_2\cdot\mu_1),
\end{eqnarray*}
and the proof is finished.
 \end{proof}

\begin{remark}
The notion of good $\dbar$-potentials is crucial in the construction. In order to use $dd^c$-potentials
one needs some regularity assumption instead of pseudomorphicity; cf.\ Remark~\ref{tejprulle}. One idea is to assume the
potential to have singularities of logarithmic type along the cycle. As long as $Y$ is smooth
any cycle has such a potential,  but we do not know if it is true for, say, nice cycles 
on a singular space.
 \end{remark}

\begin{remark}
We have no example of a cycle in $Y$ that is not nice but has a good potential;
but actually the theory only relies on existence of good potentials and some additional
condition, ensuring that the product so defined is a cycle. In case of 
nice cycles this is immediate since on germ level the intersection product is a cycle by definition.
\end{remark}

\section{RE-cycles}\label{REsektion}

Let $Y$ be a reduced analytic space of pure dimension $n$, and $\J\subset \mathcal{O}_Y$ a locally complete intersection ideal sheaf
with zero set $Z$ and codimension $\kappa$. Let $\pi\colon \widetilde Y\to Y$ be the normalization of the blowup along $\J$, let
$D$ be the exceptional divisor, and $L$ the corresponding line bundle. It is well-known, see, 
e.g., \cite[Ch.~1 and 4]{Fult},
that if $Y$ is smooth, then 
the (Lelong current of the) fundamental cycle $\mu_\J$ of $\J$ satisfies
\begin{equation}\label{fundcykel}
\mu_\J = \pi_* \big([D]\wedge \hat c_1(L^*)^{\kappa-1}\big),
\end{equation}
where $\hat c_1(L^*)$ is the first Chern form of $L^*$ with respect to some arbitrary Hermitian metric. If $Y$ is not smooth we take
\eqref{fundcykel} as the definition of the fundamental cycle of $\J$. 
It is well-known that \eqref{fundcykel} is independent of the choice of metric on $L$; it follows for instance from the dimension principle
since Chern forms of different metrics in particular are $\dbar$-cohomologous. It is also well-known that \eqref{fundcykel} is an effective
integral cycle, cf., e.g., Example~\ref{prog2} below.

The sheaf $\J$ being a locally complete intersection precisely means that the analytic subspace of $Y$ with structure sheaf 
$\mathcal{O}_Y/\J$ is a regular embedding. We say that a cycle $\mu$ in $Y$ is a \emph{regular embedding cycle}, an RE-cycle, if
$\mu$ is a locally finite sum $\sum_k a_k\mu_{\J_k}$, where $a_k\in\mathbb{Q}_+$ and $\J_k$ are locally complete intersection 
ideals.

\begin{ex}\label{prog2}
Assume that $f=(f_1,\ldots,f_\kappa)$ is a tuple of holomorphic functions generating $\J$
so that $f$ is a regular sequence at each $x\in Z$, or more generally, $f$ is a section of a Hermitian
vector bundle of rank $\kappa$ such that $f$ generates $\J$. By \cite[Proposition~1.5]{aeswy1} then 
$M_\kappa^f=\mu_\J$. It follows, see  \cite[Theorem~1.1]{aswy}, that $\mu_\J$ is an effective integral cycle. 
Notice also that $\dbar m^f_\kappa=0$ outside $Z(f)$ so that, cf.~\eqref{re3},
$m_\kappa^f$ is a good potential of $\mu_\J$.

Suppose that $\kappa=1$. 
If $Y$ is normal, then 
the fundamental cycle $\mu_\J$ is 
the divisor, $\div f$, of $f$. If $Y$ is not normal, then $\mu_\J=\pi_*(\div \pi^*f)$, where $\pi\colon \widetilde Y\to Y$ is the normalization,
and we take $\mu_\J$ as the definition of $\div f$. 
Notice that $m^f_1=\partial\log|f|^2 /2\pi i$ is a good potential of
$\text{div}\, f$.
\end{ex}

If the tuple $f$ generates $\J$ it is sometimes convenient to write $\mu_f$ rather than $\mu_\J$. 

\begin{ex}\label{punkt2}
Any point in $Y$ is an RE-cycle in view of Example~\ref{punkt}.
\end{ex}

\begin{prop}\label{lunch}
If $\J$ is a complete intersection ideal at $x$, then $\mu_\J$ is a nice cycle.
\end{prop}

\begin{proof}
Assume that $f=(f_1,\ldots, f_\kappa)$ is a minimal generating tuple for $\J$ at $x$
so that $M_\kappa^f=\mu_\J$.  
Let $i\colon Y\to Y'$ be 
an embedding and let $F=(F_1,\ldots,F_\kappa)$ be a tuple of holomorphic functions at $i(x)\in Y'$ such that 
$f=i^*F$. 
Since $\{F=0\}\cap i_*Y=i_*\{f=0\}$ has codimension $\kappa+\text{codim}_{Y'} Y$ it follows from \eqref{codimolikhet}
that $\text{codim}\, \{F=0\}=\kappa$.
Hence, $F$
defines a regular embedding at $i(x)$ and $\mu_F=M^F_\kappa$ intersects
$i_*Y$ properly in $Y'$.
Since $f=i^*F$ it follows from \eqref{morr} that $i_* M^f_\kappa=M^F_\kappa\w i_*Y$.
By Example~\ref{goodex} and Lemma~\ref{apa} thus
\begin{equation}\label{onsdag}
i_*\mu_\J=i_* M^f_\kappa=M^F_\kappa\w i_*Y=M^F_\kappa\cdot_{Y'} i_*Y = \mu_F\cdot_{Y'}i_*Y.
\end{equation}
We conclude that $\mu_\J$ is a nice cycle in $Y$.
\end{proof}

Let us illustrate the connection between the intrinsic fundamental cycle $\mu_\J$ and the representing fundamental
cycle $\mu_F$ in $Y'$ given by \eqref{onsdag} with the following example. 

\begin{ex}
Let $Y=\{z^2=w^3\}\subset\mathbb{C}^2_{z,w}$, $i\colon Y\to\mathbb{C}^2$ the inclusion, and $p=0\in Y$; cf.\ Example~\ref{spetsigt}.
Then $\pi\colon \mathbb{C}\to Y$, $\pi(t)=(t^3,t^2)$, is the normalization and, cf.\ Example~\ref{prog2},
$$
\div(z|_Y)=\pi_*\div\, t^3=3p,\quad\quad \div(w|_Y)=\pi_*\div\, t^2=2p.
$$
Thus, $3p$ and $2p$ are the fundamental cycles of the regular embeddings defined by $\langle z|_Y\rangle$ and $\langle w|_Y\rangle$,
respectively. From the extrinsic viewpoint, the fundamental cycles $\div\, z=\{z=0\}$ and $\div\, w=\{w=0\}$ are representatives 
in $\mathbb{C}^2$ of $3p$ and $2p$, respectively,
since $\{z=0\}\cdot i_*Y=3p$ and $\{w=0\}\cdot i_*Y=2p$.
\end{ex} 

\begin{prop}\label{kung}
If $\J$ and $\widetilde\J$ define properly intersecting regular embeddings, then $\J+\widetilde\J$ defines a
regular embedding  and  
$
\mu_{\J+\widetilde\J}=\mu_\J\cdot \mu_{\widetilde\J}.
$
\end{prop}

\begin{proof}
If $Y$ is smooth this is well-known and follows from, e.g., \cite{Fult}; cf.\ also Remark~\ref{truls} below.
The general case can be reduced to that case as follows.
As in the proof of Proposition~\ref{lunch}, choose minimal tuples $f$ and $\tilde f$ generating $\J$ and $\widetilde\J$, respectively,
at $x$, a local embedding $i\colon Y\to Y'$, and 
$F$ and $\widetilde F$ in $Y'$ such that $f=i^*F$ and $\tilde f=i^* \widetilde F$. Then $F$ and $\widetilde F$ define regular embeddings 
and by \eqref{onsdag}, $i_*\mu_\J = \mu_F\cdot i_*Y $ and $i_*\mu_{\widetilde\J} = \mu_{\widetilde F}\cdot i_*Y$.

Since $\mu_\J$ and $\mu_{\widetilde\J}$ intersect properly it follows that $(f,\tilde f)$, which generates $\J+\widetilde \J$, is a regular sequence
and that $\mu_F$, $\mu_{\widetilde F}$, $i_*Y$ intersect properly; cf.\ Lemma~\ref{ryydiger}. 
In particular, $\mu_F$ and $\mu_{\widetilde F}$ intersect properly in 
the smooth space $Y'$ and so,
in view of \eqref{progsnitt100} and \eqref{morr},
\begin{equation*}
i_*(\mu_\J\cdot\mu_{\widetilde\J}) = \mu_F\cdot\mu_{\widetilde F}\cdot i_*Y = 
\mu_{F,\widetilde F}\cdot i_*Y = M_{\kappa+\tilde\kappa}^{F,\widetilde F}\wedge i_*Y
=i_*(M_{\kappa+\tilde\kappa}^{f,\tilde f}) = i_*\mu_{\J+\widetilde\J},
\end{equation*}
where $\kappa$ and $\kappa'$ are the codimensions of $\J$ and $\J'$, respectively. This finishes the proof.
\end{proof}

By $\mathbb{Q}_+$-linearity, cf.\ \eqref{skorsten}, we have

\begin{cor}
Any RE-cycle in $Y$ is nice, and if 
$\mu_1$ and $\mu_2$ are RE-cycles that intersect properly,
then $\mu_1\cdot\mu_2$ is an RE-cycle.
\end{cor}

\begin{remark}\label{truls}
If $f$ and $\tilde f$ are minimal tuples that defines $\J$ and $\widetilde\J$, respectively, then
Proposition~\ref{kung} says that $(f,\tilde f)$ is a regular sequence and
\begin{equation}\label{alban2}
M^{f,\tilde f}_{\kappa+\tilde\kappa}=M^f_\kappa\w M^{\tilde f}_{\tilde\kappa},
\end{equation}
where $\kappa$ and $\kappa'$ are the codimensions of $\J$ and $\J'$, respectively. 
This may have independent interest. If $Y$ is smooth, then \eqref{alban2} is known and follows from,
e.g., \cite[Eq.~(7.6)]{aswy}.
\end{remark}

\begin{ex}[$\mathbb{Q}$-Cartier divisors]\label{ruta}
Assume that $Y$ is normal. A cycle $\mu$ in $Y$ is a $\mathbb{Q}$-Cartier divisor if locally there is a meromorphic function $f=g/h$ and a
positive integer $q$ such that $q\mu=\text{div}\, f:=\div g - \div h$.
If $\mu$ is effective it follows that $f$ is holomorphic on $Y_{reg}$, and therefore holomorphic on $Y$ by normality.
Thus effective $\mathbb{Q}$-Cartier divisors
are RE-cycles.
Assume that $\mu_1$ and $\mu_2$ are effective $\mathbb{Q}$-Cartier divisors that intersect
properly and suppose that  $\div f_j=q_j \mu_j$, $j=1,2$, for some $q_j\in\mathbb{Z}_+$. 
Then 
\begin{equation}\label{gula}
\mu_2\cdot\mu_1=\mu_2\w \mu_1=\frac{1}{q_1 q_2}\div f_2\w\div f_1.
\end{equation}
\end{ex}

\section{Global intersection formulas}\label{global}

We begin with the proof of Theorem~\ref{main2}.

\begin{proof}[Proof of Theorem~\ref{main2}]
By assumption, $\mu_1,\ldots,\mu_k$ intersect properly and we let $\nu_k=\mu_k\cdots\mu_1$, $k=1,\ldots,r$.
We inductively define currents $A_k$ of bidegree $(\kappa_1+\ldots +\kappa_k, \kappa_1+\ldots +\kappa_k-1)$
such that 
\begin{equation}\label{polynesien}
\dbar A_k=
\nu_k- \alpha_k\w\cdots\w\alpha_1.
\end{equation}
Let $A_1= a_1$. 
Assume now that $A_k$ is found. We then define
$$
A_{k+1}:= a_{k+1} \w \nu_k
+ A_k\w\alpha_{k+1}.
$$
The second product is not problematic since $\alpha_{k+1}$ is smooth. We claim that the first product, a~priori defined outside $|\mu_{k+1}|$,
has a unique pseudomeromorphic extension to $Y$, denoted by $a_{k+1} \w \nu_k$ as well, and that
\begin{equation}\label{antarktis}
\dbar(a_{k+1} \w\nu_k)=\mu_{k+1}\w\nu_k-\alpha_{k+1}\w\nu_k.
\end{equation}
Taking the claim for granted for the moment, using that $\dbar\alpha_{k+1}=0$ and \eqref{polynesien} we get 
\begin{eqnarray*}
\dbar A_{k+1} &=& (\mu_{k+1}-\alpha_{k+1}) \w \nu_k+(\nu_k- \alpha_k\w\cdots\w\alpha_1)
\w\alpha_{k+1}\\
&=&
\mu_{k+1}\w\nu_k- \alpha_{k+1}\w\cdots\w\alpha_1.
\end{eqnarray*}
Hence, $A_{k+1}$ has the desired properties.
For degree reasons,
$$
d (A_r\w \omega^{n-\kappa})=\dbar (A_r\w \omega^{n-\kappa}) = \nu_r\w\omega^{n-\kappa} - \alpha_r\w\cdots\w\alpha_1\w\omega^{n-\kappa}
$$ 
and thus
\eqref{pucko} follows by Stokes' theorem.

It remains to show the claim. 
The uniqueness is clear in view of the dimension principle since $\mu_1,\ldots,\mu_{k+1}$ intersect properly.
For the existence it is thus sufficient to check that $a_{k+1} \w \nu_k$ can be extended across $|\mu_{k+1}|$
locally in $Y$.
Let $\gamma$ be a  local smooth $\dbar$-potential of $\alpha_{k+1}$. Then
$$
a_{k+1}=a_{k+1}+\gamma -\gamma=:u_{k+1}-\gamma,
$$
and  in view of \eqref{rutger} thus $u_{k+1}$ is a local good potential of $\mu_{k+1}$. Therefore,
\begin{equation}\label{arktis}
u_{k+1}\w \nu_k-\gamma\w \nu_k
\end{equation}
is defined in view of Section~\ref{frame} and is a local \pmm extension of $a_{k+1} \w \nu_k$ across $|\mu_{k+1}|$.
Checking \eqref{antarktis} is also a local problem. We can thus replace $a_{k+1}\w\nu_k$ by \eqref{arktis} and then \eqref{antarktis}
follows by \eqref{totem15}. This finishes the proof of the claim and Theorem~\ref{main2}.
\end{proof}

In the rest of this section, given a nice cycle $\mu$ in $Y$, we consider two cases where 
there are  $a$ and $\alpha$ as in Theorem~\ref{main2}; cf.\ \eqref{rutger}.

\begin{prop}\label{gillet-soule}
Suppose that $Y$ is compact and let $\mu$ be a nice cycle in $Y$ of pure codimension $\kappa$. 
Assume that there is a global embedding $i\colon Y\to Y'$ 
into a compact K\"ahler manifold
$Y'$ and an effective cycle $\mu'$ in $Y'$ intersecting $i_*Y$ properly such that $i_*\mu=\mu'\cdot_{Y'} i_*Y$. Then
there is a \pmm current $a$ in $Y$, smooth
in $Y\setminus |\mu|$, and a smooth closed $(\kappa,\kappa)$-form $\alpha$ such that
$\dbar a=\mu-\alpha$. Moreover, $\alpha$ locally has smooth $\dbar$-potentials. 
\end{prop}
\begin{proof}
By
\cite[Theorem~1.3.5]{GS}, there is 
a smooth form $\alpha'$ in $Y'$ and a $dd^c$-potential $v'$ such that
$$
dd^c v'=\mu'-\alpha',
$$
$v'$ is smooth in $Y'\setminus |\mu'|$ and has singularities of logarithmic type
along $|\mu'|$. This latter property means the following: There is a proper surjective mapping $p\colon Y''\to Y'$   
and a current $v''$ in $Y''$ such that $v'=p_*v''$, $v''$ is smooth in $Y''\setminus p^{-1}|\mu'|$, and 
in suitable local coordinates $s=(s_1,\ldots, s_{N''})$ in $Y''$, 
$$
v''=c_1\log|s_1|^2 +\cdots +c_k\log|s_k|^2 +b,
$$ 
where $c_j$ are smooth closed forms  and $b$ is a smooth form. 

Notice that $a':=\partial v'/2\pi i$ is smooth in $Y'\setminus |\mu'|$ and $\dbar a'=\mu'-\alpha'$.
We claim that, in addition, $a'$ is  \pmm in $Y'$. In fact, 
$a'=p_* a''$ where $a'':=\partial v''/2\pi i$ locally has the form
$$
c_1\w \frac{ds_1}{2\pi i s_1} +\cdots + c_k\w \frac{ds_k}{2\pi i s_k}+\partial b/2\pi i.
$$
Thus $a''$ is \pmm in $Y''$ and it follows from \cite[Corollary~2.16]{AW3} that $a'$ is pseudomeromorphic
in $Y'$. 

Let $\alpha:=i^*\alpha'$. As in the proof of Proposition~\ref{trunk} it follows that there is a unique \pmm current $a$ in $Y$ such that
$a=i^*a'$ in $Y\setminus |\mu|$. To see that $\alpha$ and $a$ have the desired properties,
let $\gamma'$ be a local smooth $\dbar$-potential of $\alpha'$. Then $i^*\gamma'$ is a local smooth $\dbar$-potential of $\alpha$.
Moreover, $u':=a'+\gamma'$ is a local good potential of $\mu'$.
By Proposition~\ref{trunk} there is a (unique) local good potential $u$ of $\mu$ such that $u=i^*u'$ outside $|\mu|$.
Thus $u-i^*\gamma'$ is \pmm and equal to $i^*a'$ outside $|\mu|$. Hence, $a=u-i^*\gamma'$ locally, and it follows that
$\dbar a=\mu-\alpha$.
\end{proof}

We will now show that if $\mu=\mu_\J$ is the fundamental cycle of a locally complete intersection ideal $\J$
generated by a global holomorphic section of a vector bundle of rank $\kappa=\text{codim}\,\J$,
then there are $a$ and $\alpha$ as in Theorem~\ref{main2}; cf.\ Example~\ref{bongo} above and
Proposition~\ref{snabel} below. 
To find these $a$ and $\alpha$ we cannot use Proposition~\ref{gillet-soule} since in general one cannot expect that
$\mu$ has a global representative in a smooth ambient space.
Instead we will use the following generalization of the Poincar\'e--Lelong formula.
 
 \begin{prop}\label{truxa}
Let $Y$ be a reduced pure-dimensional analytic space and let 
$\sigma$ be a holomorphic section of a Hermitian vector bundle $E\to Y$ (of any rank) such that
the zero set $Z=Z(\sigma)$ has codimension $\kappa>0$.
Let $S$ denote the trivial line bundle over $Y\setminus Z$ so that
$$
0\to S\stackrel{\sigma}{\to} E\to Q\to 0
$$
is exact in $Y\setminus Z$.  If $S$ and $Q$ are equipped with the induced Hermitian metrics,
then the associated Chern forms $c(S)$ and $c(Q)$ have (unique)
locally integrable closed extensions
$C(S)$ and $C(Q)$ to $Y$. Moreover, $\log|\sigma|^2 \cdot C(Q)$ is locally integrable,
$dd^c(\log|\sigma|^2 \cdot C(Q))$ has order $0$, and
there is a real locally integrable form $W$ such that 
\begin{equation}\label{gpl}
dd^c W=M^{Q,\sigma} - c(E)+C(Q),
\end{equation}
where 
\begin{equation}\label{prat}
M^{Q,\sigma}=\1_Z dd^c(\log|\sigma|^2\cdot C(Q)).
\end{equation}
Moreover, 
 \begin{equation}\label{tomte}
M^{Q,\sigma}_\kappa = M^\sigma_\kappa.
\end{equation}
\end{prop}

Proposition~\ref{truxa} is precisely Theorem~1.1 in \cite{Apl}. However,
since the formula in \cite{Apl} is formulated only when $X$ smooth, we 
give a proof here.

\begin{proof} 
Let $\pi\colon X\to Y$ be a modification such that $X$ is smooth and the ideal generated by $\pi^*\sigma$
is principal, cf.\ the proof of Lemma~\ref{benin}. Then
$\pi^*\sigma=\sigma^0 \sigma'$, where $\sigma^0$ is a section of a line bundle $L\to X$
defining a divisor $D$ and $\sigma'$ is a non-vanishing section of $\pi^*E\otimes L^*$.
We then have the exact sequence
\begin{equation*}
0\to L\stackrel{\sigma'}{\longrightarrow} \pi^*E \longrightarrow Q'\to 0.
\end{equation*}
Equip $L$ and $Q'$ with the induced metrics and notice that $|\sigma^0|_L=|\sigma^0\sigma'|=|\pi^*\sigma|$.
Outside $\pi^{-1}(Z)$ we thus have: 

\noindent (a) $\pi^*S\stackrel{\sigma^0}{\longrightarrow} L$ is an isomorphism of 
Hermitian line bundles,

\noindent (b) the mapping $\pi^*S\stackrel{\pi^*\sigma}{\longrightarrow} \pi^*E$ factorizes
as 
$
\pi^*S\stackrel{\sigma^0}{\longrightarrow} L \stackrel{\sigma'}{\longrightarrow} \pi^*E.
$

\smallskip
By (a) we get that
$
c(S)=\pi_*c(\pi^*S)=\pi_*c(L)
$
outside $Z$, and thus $C(S):=\pi_*c(L)$ is a locally integrable closed extension of  $c(S)$.
From (b) it follows that, outside $Z$, $\pi^*Q=Q'$ as Hermitian bundles. As with $c(S)$ we see that
$C(Q):=\pi_*c(Q')$ is a locally integrable closed extension of  $c(Q)$.
Since $\log|\sigma|^2\cdot c(Q)=\pi_*(\log|\pi^*\sigma|^2\cdot c(Q'))$ outside $Z$ and
$\log|\pi^*\sigma|^2\cdot c(Q')$ is locally integrable in $X$ it follows that $\log|\sigma|^2\cdot c(Q)$ is locally integrable
in $Y$. Moreover, since $dd^c\log|\pi^*\sigma|^2$ has order $0$ and $c(Q')$ is smooth and closed
it follows that $dd^c(\log|\sigma|^2\cdot C(Q))=\pi_* dd^c(\log|\pi^*\sigma|^2\cdot c(Q'))$ has order $0$.
%
%

By \cite{BC} there is a smooth form $v$ in $X$ such that  
\begin{equation}\label{tut2}
dd^c v=c(\pi^*E)-c(L)\wedge c(Q').
\end{equation}
Moreover, by the Poincar\'e--Lelong formula on $X$, 
\begin{equation}\label{tut3}
dd^c\log|\pi^*\sigma|^2=dd^c\log|\sigma^0|_L^2=[D] -c_1(L).
\end{equation}
If  
\begin{equation}\label{tut4}
w:=\log|\pi^*\sigma|^2 c(Q')-v,
\end{equation}
then a simple calculation gives 
\begin{equation}\label{tut5}
dd^c w=
[D]\w c(Q')-c(\pi^*E)+c(Q').
\end{equation}
Since $\log|\sigma|^2\cdot c(S)$ is locally integrable, in view of \eqref{bara1} and \eqref{tut3} we have
\begin{eqnarray}\label{tut6}
M^{Q,\sigma} &=& \pi_*\big(\1_{\pi^{-1}Z}dd^c(\log|\pi^*\sigma|^2\cdot c(Q'))\big)
=\pi_*\big(\1_{|D|}([D]-c_1(L))\w c(Q')\big) \nonumber \\
&=&
\pi_*([D]\w c(Q')).
\end{eqnarray}
Thus we get \eqref{gpl} with $W:=\pi_*w$ after applying 
 $\pi_*$ to \eqref{tut5}. 

To see \eqref{tomte} notice that \eqref{tut2} gives  
$$
c(Q') = s(L)\w c(\pi^*E) - dd^c(s(L)\wedge v),
$$
where\footnote{$s(L)$ is the total Segre form of $L$.}
$
s(L):=1/c(L)=\sum_{k=1}^\infty (-c_1(L))^{k-1}.
$
Hence,
\begin{equation}\label{sladdare}
M^{Q,\sigma} = \pi_*\big([D]\wedge c(Q')\big) = 
\pi_*\big([D]\wedge s(L)\w c(\pi^*E)\big) - dd^c \pi_*\big([D]\wedge s(L)\wedge v\big).
\end{equation}
On the other hand, by Lemma~\ref{benin} and \eqref{ghana},
\begin{eqnarray*}
M^\sigma &=& \sum_{k\geq 1} M^\sigma_k = \sum_{k\geq 1}\1_{Z}\dbar m_k^\sigma 
= \sum_{k\geq 1}\pi_*\left(\1_{|D|} \dbar\left(\frac{\partial\log|\sigma^0|^2_L}{2\pi i}\w (-c_1(L))^{k-1}\right)\right) \\
&=& 
\pi_*\big(\1_{|D|} ([D]-c_1(L))\w s(L)\big)
= \pi_*\big([D]\w s(L)\big).
\end{eqnarray*}
In view of \eqref{sladdare} it thus follows that 
\begin{equation}\label{laddare}
M^{Q,\sigma} =
c(E)\wedge M^\sigma + dd^c\gamma, 
\end{equation}
where $\gamma=-\pi_*([D]\wedge s(L)\wedge v)$. Clearly, $\gamma$ has support in $Z$, and since $\pi$ is a modification it is pseudomeromorphic.
Taking the component of \eqref{laddare} of bidegree $(\kappa,\kappa)$ we obtain \eqref{tomte} since 
$M_k^\sigma=0$ for $k<\kappa$
 and the component of bidegree $(\kappa-1,\kappa-1)$ of $\gamma$ vanishes by the dimension principle.
 \end{proof}

\begin{remark}
If $\kappa=1=\text{rank}\, E$, then \eqref{gpl} is the Poincar\'e--Lelong formula,
albeit the underlying space is possibly non-normal; cf.\ Example~\ref{prog2} and \cite[Proposition~2.1]{aeswy1}.
\end{remark}

\begin{cor}\label{aplcor}
If $A=\partial W/2\pi i$, then $A$ is pseudomeromorphic, smooth outside $Z$, and
\begin{equation}\label{gpldbar}
\dbar A=dA=M^{Q,\sigma} - c(E)+C(Q).
\end{equation}
\end{cor}

\begin{proof}
Notice that $W=\pi_* w$, cf.\ \eqref{tut4}. Thus, $2\pi i A$ is $\pi_*$ of 
$$
\frac{\partial |\pi^*\sigma|^2}{ |\pi^*\sigma|^2} \w c(Q')-\partial v=
\frac{\partial |\sigma^0|^2}{ |\sigma^0|^2}\w c(Q')-\partial v
$$
and hence pseudomeromorphic since $\pi$ is a modification and $v$ is smooth. 
Now \eqref{gpldbar} follows from \eqref{gpl}.
\end{proof}

\begin{prop}\label{snabel}
Assume that  $\mu_\J$ is the fundamental cycle of a locally complete intersection ideal $\J$
of codimension $\kappa$ generated by a global holomorphic section $\sigma$ of a Hermitian vector bundle $E\to Y$ of rank $\kappa$.
Then there is a \pmm current $a$ in $Y$, smooth in $Y\setminus |\mu_\J|$, such that
\begin{equation}\label{gplre}
 \dbar a  =da =  \mu_\J - c_\kappa(E).
  \end{equation}
 \end{prop}
 
 \begin{proof}
With the notation above, let $a=A_{\kappa,\kappa-1}$. Then $a$ is \pmm and smooth outside $|\mu_\J|$.
By \cite[Proposition~1.5]{aeswy1}, $\mu_\J=M^\sigma_\kappa$, and 
since $\rank Q=\kappa-1$  we have $C_\kappa(Q)=0$. Taking the component of bidegree $(\kappa,\kappa)$
of \eqref{gpldbar} thus \eqref{gplre} follows from 
\eqref{tomte}.
\end{proof}

If $\mu=\sum_\ell q_\ell\mu_{\J_\ell}$, $q_\ell\in\mathbb{Q}_+$, 
is an RE-cycle where each $\mu_{\J_\ell}$ is as in the proposition, then by $\mathbb{Q}_+$-linearity, cf.\ \eqref{skorsten}, 
we get a \pmm $a$ and a smooth $\alpha$ such that $\dbar a=\mu-\alpha$ and $\alpha$ locally has smooth $\dbar$-potentials.

\section{Examples}\label{exsection}

\begin{ex}\label{falkland}
If $Y$ is smooth and $\mu$ is any cycle, then the proper intersection of $\mu$ and $Y$
is $\mu$, since $\mu$ is the intersection of $\Delta$ and $\mu\times Y$ in $Y\times Y$. 

In case $Y$ is singular this holds if $\mu$ is a nice cycle; recall from Example~\ref{blyerts} that 
$Y$ is nice. To see this, let  $i\colon Y\to Y'$ be a local embedding into a smooth $Y'$ and let
$\mu'$ be a representative of $\mu$ in $Y'$. Since $Y'$ is a representative of $Y$ in $Y'$ we have
$i_*(\mu\cdot Y)= \mu'\cdot_{Y'} Y'\cdot_{Y'} i_*Y= \mu'\cdot_{Y'} i_*Y=i_*\mu$ by \eqref{progsnitt100},
and thus $\mu\cdot Y=\mu$.
One can
also choose a good potential $u$ of $\mu$ and notice that 
$\mu\cdot Y=\dbar(u\w 1)=\dbar u=\mu$; cf.\ \eqref{totem15}.  
\end{ex}

\begin{ex}
Assume that $\mu_1$ and $\mu_2$ are nice in $Y$ and intersect properly.  Moreover, assume that $\mu_1=|\mu_1|=:Z$, 
let $\iota\colon Z\to Y$ be the inclusion, and let $\tau$ be the cycle in $Z$ such that 
$\iota_*\tau=\mu_2\cdot Z$. 
We first claim that $\tau$ is nice. 

Let $i\colon Y\to Y'$ be a local embedding and let $\mu'_2$ and $Z'$ be representatives of $\mu_2$ and $Z$, respectively, in $Y'$.
Since $\mu_2$ and $Z$ intersect properly, $\mu'_2$, $Z'$, $i_*Y$ intersect properly. We have
$$
(i\circ\iota)_*\tau = i_*(\mu_2\cdot Z)=\mu'_2\cdot_{Y'} Z'\cdot_{Y'} i_*Y
= \mu'_2\cdot_{Y'} i_*Z=\mu'_2\cdot_{Y'} (i\circ\iota)_*Z,
$$
where we in the second last equality consider $Z$ as a cycle in $Y$, and in the last equality as a cycle in $Z$.
Hence, $\mu'_2$ is a representative of $\tau$ in $Y'$ and so $\tau$ is nice.

Let us also notice that if $u'_2$ is a local good potential of $\mu'_2$, then by Proposition~\ref{trunk} there is a unique local good potential
$u$ of $\tau$ such that $u=\iota^*i^*u'_2$ outside $|\tau|$. In particular, if $\mu_2=\mu_f$ is the fundamental cycle of a
locally complete intersection ideal of codimension $\kappa$ generated by a holomorphic $\kappa$-tuple $f$, then 
$m_\kappa^{\iota^*f}$ is a good potential of $\tau$; cf.\ Example~\ref{prog2} and the proof of Proposition~\ref{lunch}. 
In this case thus $\tau$ is the fundamental cycle $\mu_{\iota^*f}$ and
\begin{equation}\label{vita}
\mu_f\cdot Z=\iota_* \mu_{\iota^*f}.
\end{equation}
\end{ex}

\begin{ex}\label{brazil}
Let $Y=\{xy-z^2=0\}\subset \Pr^3_{[x_0,x,y,z]}$ and let $i\colon Y\to\mathbb{P}^3$ be the inclusion. 
Then $Y$ has an isolated singular point $p=[1,0,0,0]$; the 
so-called $A_1$-singularity. Consider the lines 
$$
L_1=\{[x_0, x,y,z];\  x=z=0\}, \quad L_2=\{[x_0, x,y,z]; \  y=z=0\}.
$$
It is straightforward to check that the sections $\sigma_1=i^*x$ and $\sigma_2=i^*y$ of $i^*\Ok(1)$ vanish to order $2$ 
on $L_1\setminus \{0\}$ and $L_2\setminus \{0\}$, respectively. Thus, $\div \sigma_1= 2L_1$ and $\div \sigma_2 = 2L_2$,
so $L_1$ and $L_2$ are $1/2$-Cartier divisors.
Let $\iota\colon L_1\to Y$ be the inclusion. Noticing that $\iota^*\sigma_2=y|_{L_1}$, by \eqref{vita} we have that 
\begin{equation}\label{expected}
2L_2\cdot L_1=\iota_* \div \iota^*\sigma_2 = [p],
\end{equation}
and so $L_2\cdot L_1=(1/2)[p]$.
\end{ex}

To see that $L_2\cdot L_1=(1/2)[p]$ one can also use Theorem~\ref{main2} as follows. 

\begin{ex}
We keep the notation from Example~\ref{brazil} and let $\omega$ be the Fubini--Study metric form on $\mathbb{P}^3$.
By the Poincar\'e--Lelong formula we have  $dd^c\log|\sigma_j|^2=\div\sigma_j-i^*\omega$,
and by Theorem~\ref{main2} thus
$$
\int_Y 2L_2\cdot 2L_1 = \int_Y\div \sigma_2\wedge\div\sigma_1 = \int_Y\omega^2.
$$
Now,  $\int_Y \omega^2 = 2$ since $Y$ has degree $2$, and hence $\int_Y L_2\cdot L_1 = 1/2$. Since $|L_2\cdot L_1|=\{p\}$
it follows that $L_2\cdot L_2=(1/2)[p]$.

We remark that there is nothing special about the lines $L_1$ and $L_2$. In fact, the set of pairs of lines in $Y$ through $p$ 
(including double lines) is in one-to-one correspondence with the set of divisors, containing $p$, of sections of $i^*\mathcal{O}(1)$.
\end{ex}

The next example shows that \eqref{urk} does not hold in general if the representative $\mu'$ of $\mu$ in Definition~\ref{nicedef}
is not effective.

\begin{ex}
We continue to keep the notation of Example~\ref{brazil} and let $\sigma_3=i^*z$. Then $\{\sigma_3=0\}=L_1\cup L_2$ and
$\sigma_3$ vanishes to order $1$ along $L_1\cup L_2\setminus\{0\}$. Hence, $\div\, \sigma_3=L_1+L_2$. 
Since $\div\, \sigma_2=2L_2$ thus $\div\,\sigma_3 - (1/2)\div\, \sigma_2 = L_1$. It follows that 
$\mu':=\div\, z - (1/2)\div\, y$ is a ``representative'' of $L_1$ in $\mathbb{C}^3$ such that 
$|\mu'|\cap Y=L_1\cup L_2$ strictly contains $L_1$.
\end{ex}

\begin{ex}\label{Hartshorne}
Let $Z_1$ and $Z_2$ be two $2$-dimensional planes in $\C^4$. Clearly, $Z_j$ are fundamental cycles of complete intersection
ideals so the cycle $Z_1+Z_2$ is an RE-cycle. However, if $Z_1$ and $Z_2$ intersect properly, so that $Z_1\cap Z_2$ is just a point $p$, 
then no multiple of $Z_1+Z_2$ is 
the fundamental cycle of a complete intersection. This follows from Hartshorne's connectedness theorem, which says that a 
set-theoretic complete intersection
is connected in codimension $1$. Since $Z_1\cup Z_2$ is $2$-dimensional and 
becomes disconnected by removing $p$ thus $Z_1\cup Z_2$ is not a complete intersection.

If $Z_1$ and $Z_2$ do not intersect properly, then $Z_1+Z_2$ is the fundamental cycle of a complete intersection ideal. This is clear 
since then either 
$Z_1=Z_2$ or $Z_1\cap Z_2$ is a line.
\end{ex}

Here is an example of a singular $Y$ where similar phenomena occur.

\begin{ex}\label{god-ej-RE}
Let $f\colon \mathbb{C}_z^4\to\mathbb{C}_w^{10}$ be the mapping 
\begin{equation*}
f(z)=(z^{\alpha_1},\ldots,z^{\alpha_{10}})=(w_1,\ldots,w_{10}), 
\end{equation*}
where 
$z^{\alpha_j}$ are the monomials in $\mathbb{C}^4$ of degree $2$.
Let $Y=f(\mathbb{C}^4)$ and let $i\colon Y\to \C^{10}$ be the inclusion.
The differential of $f$ is injective outside $0$ so $Y$ is smooth outside $0=f(0)$.
One can check that $f$ is $2 : 1$ outside $0$, and using, e.g., \cite[(6.1)]{aeswy1} that 
the multiplicity of $Y$ at $0$ is $8$. In particular, $0\in Y$ is a singular point.

Let $\widetilde Z_1=\{z_1=z_2=0\}\subset \C^4$ and $Z_1=f(\widetilde Z_1)$. Choose the monomials $z^{\alpha_j}$ so that 
$z^{\alpha_1}=z_1^2$ and $z^{\alpha_2}=z_2^2$. Then, since  $f_*1=2$, by \eqref{morr} and Remark~\ref{truls} we have
\begin{eqnarray}\label{feynman}
2 M_2^{i^*w_1,i^*w_2}&=& f_*(M_2^{z_1^2,z_2^2}) = f_*(M_1^{z_1^2}\w M_1^{z_2^2})
=f_*(2[z_1=0]\w 2[z_2=0]) \\
&=&
4f_*\widetilde Z_1=8Z_1.\nonumber
\end{eqnarray}
Hence, $4Z_1$ is the fundamental cycle of $\langle i^*w_1,i^*w_2\rangle$, so $Z_1$ is in particular an RE-cycle.

If $\widetilde Z_2$ is another $2$-dimensional linear subspace of $\C^4$, then in the same way $Z_2:=f(\widetilde Z_2)$ is an RE-cycle,
and hence $Z_1+Z_2$ is an RE-cycle.
If $\widetilde Z_2$ intersects $\widetilde Z_1$ properly, then $\widetilde Z_1\cup \widetilde Z_2$ is not a complete intersection; cf.\ 
Example~\ref{Hartshorne}. In this case, since $f^{-1}(Z_1\cup Z_2)=\widetilde Z_1\cup \widetilde Z_2$, it follows that
no multiple of $Z_1+Z_2$ is the fundamental cycle of a complete intersection
ideal.  If, on the other hand, $\widetilde Z_2$ and $\widetilde Z_1$ do not intersect properly, then a multiple of
$Z_1+Z_2$ is the fundamental cycle of a complete intersection
ideal. For instance, if $\widetilde Z_2=\{z_2=z_3=0\}$ and $w_3=z^{\alpha_3}=z_1z_3$, then in a similar way as in \eqref{feynman},
$$
M_2^{i^*w_2,i^*w_3}=2(Z_1+Z_2).
$$
\end{ex}

\smallskip

In the next example we will see that one can give a meaning to the intersection $\J\cdot \mu$, where $\J$ generates a regular embedding
and $\mu$ is any cycle such that $\text{codim}\, Z(\J)\cap |\mu| = \text{codim}\, Z(\J) + \text{codim}\, |\mu|$.

\begin{ex}\label{reex}
Let $\J$ be a locally complete intersection ideal sheaf on $Y$ of codimension $\kappa$ and let $\mu$ be a
cycle in $Y$ 
such that $\text{codim}\, Z(\J)\cap |\mu| = \kappa + \text{codim}\, |\mu|$. 
Let $f=(f_1,\ldots,f_\kappa)$ be a holomorphic tuple locally generating $\J$.
We claim that 
\begin{equation*}
\J\cdot \mu := M_\kappa^f\wedge \mu,
\end{equation*}
cf.\ Section~\ref{oponPM},
is the Lelong current of a cycle with support $Z(\J)\cap |\mu|$ that only depends on the integral closure class
of $\J$. 

To see this, assume first that $\mu$ is an irreducible
subvariety and let $\iota\colon \mu\to Y$ be the inclusion. Then, by \eqref{morr} we have
$M_\kappa^f\wedge \mu=\iota_* M_\kappa^{\iota^*f}$. Since $Z(\J)$ and $\mu$ intersect properly it follows that
$\iota^*f$ is a regular sequence at all points where $\iota^*f=0$. In view of Example~\ref{prog2} thus 
$M_\kappa^{\iota^*f}$ is a cycle in $\mu$ and we conclude that $M_\kappa^f\wedge \mu$ is a cycle in $Y$.
Moreover, in view of \cite[Remark~4.1]{aswy}, $M_\kappa^{\iota^*f}$ only depends on the integral closure class of $\la\iota^*f\ra$.
The claim now follows for an arbitrary $\mu$ by linearity.
\end{ex}

If $\mu$ in the preceding example has a good potential, e.g., is a nice cycle, then in view of Section~\ref{frame},
$\J\cdot\mu = \mu_\J\w\mu=\mu\w\mu_\J$. Thus, in this case, 
$\J\cdot\mu$ only depends on the fundamental cycle $\mu_\J$ of $\J$. We will see in the following  example
that  in general, even for a principal ideal $\J$, $\J\cdot\mu$ depends on
(the integral closure class of) $\J$ and not only on its fundamental cycle.

\begin{ex}\label{tomat1}
Let $Y=\{(x,y,z)\in\C^3;\, xy=0\}$ and let, for positive integers $p$ and $q$, $f_{p,q}$ be the restriction to $Y$ of $x^p+y^q$.
Let $\pi\colon Y'\to Y$ be the normalization and notice that $Y'$ is the disjoint union of 
$Y'_1\simeq \{y=0\}$ and $Y'_2\simeq \{x=0\}$.
We have that, cf.\ Example~\ref{prog2},
\begin{equation*}
\div f_{p,q} = \pi_* \div \pi^*f_{p,q} = \pi_*(\div x^p|_{Y'_1} + \div y^q|_{Y'_2}) = (p+q)[Z],
\end{equation*} 
where $Z=\{x=y=0\}$. Thus, $(p+q)[Z]$ is the fundamental cycle of the ideal $\J_{p,q}\subset \mathcal{O}_Y$ generated by $f_{p,q}$, and so
$[Z]$ is an RE-cycle. 
Notice that
$\J_{p,q}$ and $\J_{p',q'}$ have the same fundamental cycle if $p+q=p'+q'$.

Let $[W]=\{y=z=0\}$ and let $\iota\colon W\to Y$ be the inclusion. Then, cf.\ the preceding example,
\begin{equation*}
\J_{p,q}\cdot [W] = \div f_{p,q}\wedge [W] = \iota_* \div \iota^*f_{p,q} = \iota_* \div (x^p|_W)
= p[0].
\end{equation*}
Since $\J_{1,2}$ and $\J_{2,1}$ have the same fundamental cycle it follows that the product defined in Example~\ref{reex} does not only depend on 
the fundamental cycle of $\J$. We notice also that $[W]$ cannot have a good potential; in particular it cannot be 
an RE-cycle, or even a nice cycle. Indeed, if it had, then $\J_{p,q}\cdot [W]$ would only depend 
on the fundamental cycle of $\J_{p,q}$ in view of the discussion just before this example.
\end{ex}

Assume that $Y$ has pure dimension $n$, let $\iota\colon Y\to Y\times Y$ be the diagonal embedding
and $\eta$ a tuple of holomorphic functions defining $\Delta=\iota_*Y$ in $Y\times Y$.
In view of Example~\ref{fredag} we have $M_n^{\eta}=[\Delta]$.
By \eqref{prog1} thus
\begin{equation}\label{sportextra}
M_n^{\eta}\w (\alpha\otimes\beta)=  \iota_*(\alpha\w\beta)
\end{equation}
if $\alpha$ and $\beta$ are smooth forms. If $\alpha$ and $\beta$ are good $\dbar$-potentials $u_j$ of cycles $\mu_j$, then
we used that \eqref{sportextra} holds outside $|\mu_1|\times |\mu_2|$ to define the product $u_1\w u_2$; see the proof of Proposition~\ref{baker1}.
Moreover, if $\alpha$ and $\beta$ are generically smooth and have certain mild singularities, then the left-hand side of
\eqref{sportextra} can be used to give a reasonable
meaning to the product $\alpha\w\beta$, see, e.g., \cite{RSWSerre}.
However, contrary to the case when $Y$ is smooth, if $\alpha$ and $\beta$ are cycles, then the left-hand side of
\eqref{sportextra} cannot in general be used to give a reasonable definition of $\alpha\w\beta$ for singular $Y$
as the following example shows.

\begin{ex}\label{fruktsallad}
Let $Y$ and $p$ be as in Example~\ref{brazil}, but here considered in $\C^3_{x,y,z}$
so that $p=(0,0,0)$. By Example~\ref{punkt2}, $[p]$ is an RE-cycle.
Let $\iota\colon Y\to Y\times Y$ be the diagonal embedding and $\eta$ the restriction to $Y\times Y$ of $(x-x',y-y',z-z')$
so that $\eta$ defines the diagonal $\iota_*Y$. We claim that 
\begin{equation}\label{cocos}
M_2^\eta\w (Y\otimes [p]) = 2\iota_*[p].
\end{equation}
Using the left-hand side as a definition of a proper intersection product of $Y$ and $[p]$ would thus not
be in agreement with Example~\ref{falkland}.

To see \eqref{cocos},
consider the generically $2 : 1$ mapping
$$
\pi\colon \C^2\to Y; \quad (u,v)\mapsto (u^2,v^2,uv).
$$
and let $g$ be the mapping $Y\to Y\times Y$, $\xi\mapsto (\xi,p)$.
Since $g_*1=Y\otimes [p]$ and $\pi_*1=2$, by \eqref{morr} we have that
\begin{equation}\label{cocosboll}
M_2^\eta\w  (Y\otimes [p])=g_* M_2^{g^*\eta}=g_* M_2^{x,y,z}
=\frac{1}{2}g_*\pi_* M_2^{u^2,v^2,uv}.
\end{equation}
The ideal $(u^2,v^2,uv)$ has the same integral closure as the regular sequence $(u^2,v^2)$
and hence, cf.\  \cite[Remark~4.1]{aswy},
\begin{equation}\label{cocosbollar}
M_2^{u^2,v^2,uv}=M_2^{u^2,v^2}=4[0].
\end{equation}
Since $\pi_*[0]=[p]$ thus \eqref{cocos} follows from \eqref{cocosboll} and \eqref{cocosbollar}.
\end{ex}

\section{The case when $Y$ is nearly smooth}\label{nearly}

Recently, in \cite{BMmathZ}, Barlet and Magn\'usson introduced a class of analytic spaces called \emph{nearly smooth}.
An analytic space $Y$ is nearly smooth if it is normal and if for each
$y\in Y$ there is a neighborhood $\mathcal{U}$ of $y$, a connected complex manifold $\widetilde Y$, and a 
proper holomorphic surjective finite mapping $q\colon \widetilde Y\to \mathcal{U}$. 
Such a mapping $q$ is called a local model. The number of points
of $q^{-1}(y)$ for generic $y\in Y$ is constant and denoted by $\text{deg}\, q$.

Recall that if $\mu_\J$ is the fundamental cycle of a locally complete intersection ideal $\J$ of codimension $\kappa$
generated by a tuple 
$f=(f_1,\ldots,f_\kappa)$, then $\mu_\J=M_\kappa^f$; cf.\ Example~\ref{prog2}.

\begin{lma}\label{radio}
Suppose that $f_j=(f_{j,1},\ldots,f_{j,\kappa_j})$, $j=1,2$, are holomorphic tuples such that $\text{codim}\, f_j^{-1}(0)=\kappa_j$
and that the corresponding fundamental cycles $\mu_j=M_{\kappa_j}^{f_j}$ intersect properly. 
If $q\colon\widetilde Y\to Y$ is a local model, then $M_{\kappa_j}^{q^*f_j}$ are properly intersecting fundamental cycles and
\begin{equation*}
\mu_2\cdot\mu_1 = \frac{1}{\text{deg}\, q} q_*(M_{\kappa_2}^{q^*f_2} \wedge M_{\kappa_1}^{q^*f_1}).
\end{equation*}
\end{lma}

\begin{proof}
Since $q^{-1}(y)$ is $0$-dimensional for all $y\in Y$ it follows that $M_{\kappa_j}^{q^*f_j}$ are properly intersecting 
fundamental cycles of complete intersection ideals.

Since $q_*1 = \text{deg}\, q$, by \eqref{morr} we have that 
$\text{deg}\, q \cdot M_{\kappa_1}^{f_1}=q_*M_{\kappa_1}^{q^*f_1}$.
By \eqref{morr} again thus
\begin{equation*}
\text{deg}\, q \cdot M_{\kappa_2}^{f_2}\wedge M_{\kappa_1}^{f_1}=
q_*\left(M_{\kappa_2}^{q^*f_2}\wedge M_{\kappa_1}^{q^*f_1}\right).
\end{equation*}
The lemma now follows in view of Example~\ref{goodex}.
\end{proof}

By \cite[Proposition~1.1.5]{BMmathZ}, the inequality \eqref{codimolikhet} holds for nearly smooth spaces. For such spaces 
proper intersection thus has a clear meaning, and
an intersection product $\mu_2\cap_Y\mu_1$ for any two properly intersecting cycles $\mu_1$ and $\mu_2$ in $Y$
is introduced in \cite{BMmathZ}. 
The main result of this section is the following proposition.

\begin{prop}\label{nearlyprop}
Let $Y$ be nearly smooth and let $\mu_1$ and $\mu_2$ be properly intersecting RE-cycles in $Y$.
Then $\mu_2\cdot \mu_1 = \mu_2\cap_Y \mu_1$.
\end{prop}

\begin{lma}\label{analytlemma}
Let $f=(f_1,\ldots,f_\kappa)$ be a holomorphic tuple 
in $Y$ such that $\text{codim}\, f^{-1}(0)=\kappa$.
Then there is a neighborhood $S$ of $0\in\C^\kappa$ such that $\big(M_\kappa^{f-s}\big)_{s\in S}$ is an analytic family
of cycles in $Y$ parametrized by $S$. In particular, $\lim_{s\to 0} M_\kappa^{f-s} = M_\kappa^{f}$ as currents.
\end{lma}

The definition of an analytic family of cycles can be found, e.g., in \cite[Section~4.3.1]{BMbook}.
The precise definition is not needed in this paper; instead
we will recall and use various natural properties of such families when needed.

\begin{proof}
Notice first that there is a neighborhood
$S\subset \C^\kappa$ of $0$ such that $\text{codim}\, f^{-1}(s)=\kappa$ for all $s\in S$;
see, e.g., \cite[Proposition~2.4.60]{BMbook}. 

It is a local problem in $Y$ to show that $\big(M_\kappa^{f-s}\big)_{s\in S}$ is an analytic family so we can assume that we 
have a local model $q\colon\widetilde Y\to Y$. 
As in the proof of Lemma~\ref{radio}  we have
$q_* M_\kappa^{q^*f-s} = \text{deg}\, q\cdot  M_\kappa^{f-s}$ and
$\text{codim}\, (q^*f)^{-1}(s)=\kappa$, $s\in S$.
If $\big(M_\kappa^{q^*f-s}\big)_{s\in S}$ is an analytic family in $\widetilde Y$, then $q_* M_\kappa^{q^*f-s}$ is an analytic family in $Y$ by
\cite[Theorem~4.3.22]{BMbook}. Hence, to show the lemma it suffices to show that $\big(M_\kappa^{q^*f-s}\big)_{s\in S}$ is analytic. 
We may thus assume that
$Y$ is smooth. Possibly replacing $Y$ by $f^{-1}(S)$ we may also assume that $f\colon Y\to S$ has fibers of constant codimension $\kappa$.

Let $G\subset Y\times S$ be the graph of $f$ and let $H_s:=Y\times\{s\}\subset Y\times S$. Then $G$ and $H_s$ intersect properly 
and $X_s:=G\cdot H_s$ is after the identification $H_s\simeq Y$ a cycle in $Y$. In view of \cite[Ch.\ VII, Proposition~1.5.1]{BMbookII},
$(X_s)_{s\in S}$ is an analytic family of cycles in $Y$.

Let $F(x,s)=f(x)-s$. By Example~\ref{fredag} we have $M_\kappa^F=[G]$.  Fix an arbitrary $s_0\in S$ and let 
$i\colon Y\to Y\times S$ be the embedding $x\mapsto (x,s_0)$ so that $i_*1=H_{s_0}$. 
In view of Section~\ref{Chirkaprod} or Section~\ref{Fultprod}  and \eqref{morr},
\begin{equation*}
i_* X_{s_0} = G\cdot H_{s_0} = M_\kappa^F\wedge i_*1 = i_* M_\kappa^{i^*F}=i_*M_{\kappa}^{f-s_0}.
\end{equation*}
Hence, $X_{s_0}=M_\kappa^{f-s_0}$ and it follows that $(M_\kappa^{f-s})_{s\in S}$ is an analytic family.
The last statement of the lemma follows from \cite[Proposition~4.2.17]{BMbook}.
\end{proof}

If $q\colon \widetilde Y\to Y$ is a local model and $\mu$ is a cycle in $Y$, then there is a natural pullback cycle
$q^*\mu$ in $\widetilde Y$, see \cite[Section~2.1]{BMmathZ}. We have the following corollary of Lemma~\ref{analytlemma}.

\begin{cor}\label{pullbackcor}
Let $f=(f_1,\ldots,f_\kappa)$ be as in Lemma~\ref{analytlemma},
let $q\colon \widetilde Y\to Y$ be a local model, and let $\mu=M_\kappa^f$. Then $q^*\mu=M_\kappa^{q^*f}$.
\end{cor}

\begin{proof}
Let $\widetilde V=\{x\in\widetilde Y;\, \text{rank}_x\, q < \text{dim}\, Y\}$ and $V=q(\widetilde V)$. Then,
since $q$ is proper and surjective, $\widetilde V$ and $V$ are
nowhere dense
analytic subsets of $\widetilde Y$ and $Y$, respectively.

By \cite[Section~2.1]{BMmathZ}, $q^*$ has the following two properties. First, if $Z$ is a cycle with no component contained in 
$V\cup Y_{sing}$ and $Z=|Z|$, then 
\begin{equation}\label{mustiff}
q^*Z=q^{-1} |Z|.
\end{equation} 
Second, if $(\mu_s)_{s\in S}$ is an analytic family of cycles in $Y$, then 
$(q^*\mu_s)_{s\in S}$ is an analytic family of cycles in $\widetilde Y$.

Let now $\mu_s:=M_\kappa^{f-s}$ and $\widetilde\mu_s:=(M_\kappa^{q^*f-s})_{s\in S}$, where $S$ is as in Lemma~\ref{analytlemma}. 
By that lemma, $(\mu_s)_{s\in S}$
and $(\widetilde\mu_s)_{s\in S}$ are analytic families of cycles in $Y$ and $\widetilde Y$, respectively.
We claim that 
\begin{equation}\label{klejm}
q^*\mu_s=\widetilde\mu_s
\end{equation}
for all $s\in S$, from which the corollary follows. 
To show the claim it suffices to check that \eqref{klejm} holds for generic $s$ in $S$
since both $q^*\mu_s$ and $\widetilde\mu_s$ are analytic, in particular continuous, in $s$.
Let $\widetilde A=\{x\in \widetilde Y;\, \text{rank}_x\, f\circ q <\kappa\}$ and $A=q(\widetilde A)$.
Then $\widetilde A$ and $A$ are nowhere dense analytic subsets of $\widetilde Y$ and $Y$, respectively.

For generic $s$, by, e.g., \cite[Corollary~2.4.61]{BMbook}, $\mu_s$ has no component contained in $A\cup V\cup Y_{sing}$ and $\widetilde\mu_s$
has no component contained in $\widetilde A$.
Fix such an $s$. Outside $A\cup V\cup Y_{sing}$, $f$ has constant rank $\kappa$ and so $\{f=s\}$ is a submanifold and 
$f-s$ generates its radical ideal there.  By Example~\ref{fredag} thus 
\begin{equation}\label{mastiff}
\mu_s=|\mu_s|
\end{equation}
outside $A\cup V\cup Y_{sing}$. Since $\mu_s$ has no component contained in $A\cup V\cup Y_{sing}$ it follows that \eqref{mastiff}
holds in $Y$. In the same way it follows that $\widetilde\mu_s=|\widetilde\mu_s|$. Since $\mu_s$ in particular has no component contained in 
$V\cup Y_{sing}$ it now follows from \eqref{mustiff} that
$$
q^*\mu_s=q^{-1} |\mu_s|=q^{-1}\{f=s\}=\{f\circ q=s\}=|\widetilde \mu_s|=\widetilde\mu_s
$$
and the claim and the corollary are proved.
\end{proof}

\begin{proof}[Proof of Proposition~\ref{nearlyprop}]
This is a local statement 
so after shrinking $Y$ we may assume that there is a local model $q\colon \widetilde Y\to Y$.
Moreover, by linearity we can assume that 
$\mu_j=M_{\kappa_j}^{f_j}$ for holomorphic tuples $f_j=(f_{j1},\ldots,f_{j\kappa_j})$  such that
$\text{codim}\, f_j^{-1}(0)=\kappa_j$.

The pullback cycles $q^*\mu_j$ intersect properly in $\widetilde Y$
and, by \cite[Theorem~3.1.5]{BMmathZ}, 
\begin{equation}\label{hamta0}
\mu_2\cap_Y\mu_1 = \frac{1}{\text{deg}\, q} q_*(q^*\mu_2 \cdot q^*\mu_1),
\end{equation}
where $q^*\mu_2\cdot q^*\mu_1$ is the proper intersection product in $\widetilde Y$.
By Lemma~\ref{radio} thus the proposition follows from Corollary~\ref{pullbackcor}.
\end{proof}

\begin{remark}
By \cite[Proposition~1.1.5]{BMmathZ}, any Weil divisor in a nearly smooth $Y$ is a $\mathbb{Q}$-Cartier divisor.
Effective Weil divisors in $Y$ thus are RE-cycles, see Example~\ref{ruta}.
In view of Example~\ref{falkland} and since any $0$-dimensional cycle is RE by Example~\ref{punkt2}
we thus see that for $2$-dimensional nearly smooth spaces the proper intersection product in \cite{BMmathZ} 
coincides with our product for effective cycles.
\end{remark}

\end{document}